\documentclass{amsart}

\title[Scissors automorphism groups I: Homological stability and K-theory]{Scissors automorphism groups I: \\ Homological stability and K-theory}

\author[A. Kupers]{Alexander Kupers}
\address{Department of Computer and Mathematical Sciences, University of Toronto Scarborough, 1265 Military Trail, Toronto, ON M1C 1A4, Canada}
\email{a.kupers@utoronto.ca}

\author[E. Lemann]{Ezekiel Lemann} 
\address{Department of Mathematics and Statistics, Binghamton University, PO Box 6000, Binghamton, NY 13902}
\email{elemann1@binghamton.edu}

\author[C. Malkiewich]{Cary Malkiewich}
\address{Department of Mathematics and Statistics, Binghamton University, PO Box 6000, Binghamton, NY 13902}
\email{cmalkiew@binghamton.edu}

\author[J. Miller]{Jeremy Miller} 
\address{Mathematical Sciences Building, Purdue University 150 N University St, West Lafayette, IN 47907}
\email{jeremykmiller@purdue.edu}

\author[R. J. Sroka]{Robin J. Sroka} 
\address{Mathematisches Institut, Universität Münster, Einsteinstrasse 62, 48149 Münster, Germany}
\email{robinjsroka@uni-muenster.de}

\usepackage[T1]{fontenc}
\usepackage{lmodern}
\usepackage[final]{microtype}
\usepackage[mathscr]{euscript}

\usepackage[backend=bibtex, style=alphabetic, url=false, hyperref, backref, backrefstyle=none, maxbibnames=99, sorting=nyt, safeinputenc, giveninits=true]{biblatex}
\AtEveryBibitem{\clearlist{language}}
\RequirePackage{doi}

\usepackage[dvipsnames,svgnames,x11names,hyperref]{xcolor}
\PassOptionsToPackage{hyphens}{url}\usepackage{hyperref}
\hypersetup{
	colorlinks=true,
	citecolor=Mahogany,
	linkcolor=Mahogany,
	urlcolor=Mahogany,
	filecolor=Mahogany,
	bookmarksdepth=3
}

\usepackage{multicol}
\usepackage{amssymb}
\usepackage{amsmath}
\usepackage{amsthm}
\usepackage{mathtools}
\calclayout

\usepackage{tikz}
\usepackage{tikz-cd}

\newcommand*{\figmargin}{.2in}
\newcommand*{\figwidth}{1.5in}

\usepackage{xparse}
\usepackage{etoolbox}
\usepackage{aliascnt}
\usepackage{hyperref}
\hypersetup{
	colorlinks=true,
	linkcolor=blue,
	filecolor=blue,
	citecolor = blue,
	urlcolor=blue,
}
\usepackage[capitalize,nameinlink,noabbrev]{cleveref}

\AtBeginDocument{
	\def\MR#1{}
}

\usepackage{enumitem}
\setenumerate[1]{label=(\roman*),}
\setitemize[1]{label=\raisebox{0.25ex}{\boldmath$\cdot$}}

\usepackage[textsize = footnotesize]{todonotes}
\usepackage{comment}

\newcounter{environmentcounteralphabetic}

\numberwithin{environmentcounter}{section}

\newaliascnt{definitioncounteralias}{environmentcounter}

\newaliascnt{notationcounteralias}{environmentcounter}

\newaliascnt{remarkcounteralias}{environmentcounter}

\newaliascnt{examplecounteralias}{environmentcounter}

\newaliascnt{constructioncounteralias}{environmentcounter}

\newaliascnt{lemmacounteralias}{environmentcounter}

\newaliascnt{propositioncounteralias}{environmentcounter}

\newaliascnt{corollarycounteralias}{environmentcounter}

\newaliascnt{theoremcounteralias}{environmentcounter}

\newaliascnt{questioncounteralias}{environmentcounter}

\newaliascnt{conjecturecounteralias}{environmentcounter}

\newaliascnt{figurecounteralias}{environmentcounter}

\theoremstyle{definition}
\newtheorem{definition}[definitioncounteralias]{Definition}

\newtheorem{notation}[notationcounteralias]{Notation}

\theoremstyle{plain}
\newtheorem{lemma}[lemmacounteralias]{Lemma}
\newtheorem{proposition}[propositioncounteralias]{Proposition}
\newtheorem{corollary}[corollarycounteralias]{Corollary}

\newtheorem{theorem}[theoremcounteralias]{Theorem}
\newtheorem{theoremalphabetic}[environmentcounteralphabetic]{Theorem}
\newtheorem{corollaryalphabetic}[environmentcounteralphabetic]{Corollary}
\newtheorem{conjecturealphabetic}[environmentcounteralphabetic]{Conjecture}

\theoremstyle{remark}
\newtheorem{remark}[remarkcounteralias]{Remark}
\newtheorem{example}[examplecounteralias]{Example}

\makeatletter
\def\namedlabel#1#2{\begingroup
	#2%
	\def\@currentlabel{#2}%
	\phantomsection\label{#1}\endgroup
}
\makeatother

\newcommand{\vol}{\on{vol}}
\newcommand{\cA}{\mathcal{A}}

\newcommand{\cE}{\mathcal{E}}
\newcommand{\cH}{\mathcal{H}}
\newcommand{\cW}{\mathcal{W}}
\newcommand{\cG}{\mathcal{G}}
\newcommand{\cL}{\mathcal{L}}
\newcommand{\cU}{\mathcal{U}}
\newcommand{\cX}{\mathcal{X}}
\newcommand{\bR}{\mathbf{R}}
\newcommand{\congto}{\xrightarrow{\raisebox{-.5ex}[0ex][0ex]{$\sim$}}}

\newcommand{\on}[1]{\operatorname{#1}}
\newcommand{\scr}[1]{\mathscr{#1}}
\newcommand{\colim}{\on{colim}}
\newcommand{\hocolim}{\on{hocolim}}

\newcommand{\ab}{{ab}}
\newcommand{\diam}{\on{diam}}

\newcommand{\category}[1]{\mathcal{#1}}
\newcommand{\catC}{\category{C}}
\renewcommand{\hom}{\on{Hom}}
\newcommand{\aut}{\mathrm{Aut}}
\newcommand{\id}{\mathrm{id}}

\newcommand{\ob}{\mathrm{ob}}

\newcommand{\Z}{\mathbb{Z}}

\newcommand{\Q}{\mathbb{Q}}
\newcommand{\realnumbers}{\mathbb{R}}
\newcommand{\R}{\mathbb{R}}

\newcommand{\Link}[2]{\on{Link}_{#1}(#2)}

\begin{document}
	
	\begin{abstract}
	In any category with a reasonable notion of cover, each object has a group of scissors automorphisms. We prove that under mild conditions, the homology of this group is independent of the object, and can be expressed in terms of the scissors congruence K-theory spectrum defined by Zakharevich. We therefore obtain both a group-theoretic interpretation of Zakharevich's higher scissors congruence K-theory, as well as a method to compute the homology of scissors automorphism groups. In the classical case of scissors congruence of polytopes, this leads to calculations and structural results on the homology of the scissors automorphism group. In two sequels, we show that more generally our approach leads to homology calculations for various families of groups appearing in group theory and dynamics, recovering results of Szymik--Wahl, Li, and Tanner.
	\end{abstract}

	\maketitle
	\thispagestyle{empty} 

	\tableofcontents

	\vspace{-.5cm}
	\section{Introduction}
	In this paper, we study the group homology of various scissors congruence automorphism groups, or \emph{scissors automorphism groups} for short. Informally, a scissors automorphism of an $n$-dimensional Euclidean polytope is a self-map obtained by cutting the polytope into finitely many smaller polytopes and rearranging the pieces using isometries, as shown below. In the literature this is also called an invertible piecewise isometry, see e.g.~\cite{ag_pwi_pentagon_example,bp_pwi_classification,car_pwi_construction}.
	
	\vspace{1em}
	\centerline{
	\def\svgwidth{3in}
\begingroup%
  \makeatletter%
  \providecommand\color[2][]{%
    \errmessage{(Inkscape) Color is used for the text in Inkscape, but the package 'color.sty' is not loaded}%
    \renewcommand\color[2][]{}%
  }%
  \providecommand\transparent[1]{%
    \errmessage{(Inkscape) Transparency is used (non-zero) for the text in Inkscape, but the package 'transparent.sty' is not loaded}%
    \renewcommand\transparent[1]{}%
  }%
  \providecommand\rotatebox[2]{#2}%
  \newcommand*\fsize{\dimexpr\f@size pt\relax}%
  \newcommand*\lineheight[1]{\fontsize{\fsize}{#1\fsize}\selectfont}%
  \ifx\svgwidth\undefined%
    \setlength{\unitlength}{222.38027826bp}%
    \ifx\svgscale\undefined%
      \relax%
    \else%
      \setlength{\unitlength}{\unitlength * \real{\svgscale}}%
    \fi%
  \else%
    \setlength{\unitlength}{\svgwidth}%
  \fi%
  \global\let\svgwidth\undefined%
  \global\let\svgscale\undefined%
  \makeatother%
  \begin{picture}(1,0.4391301)%
    \lineheight{1}%
    \setlength\tabcolsep{0pt}%
    \put(0,0){\includegraphics[width=\unitlength,page=1]{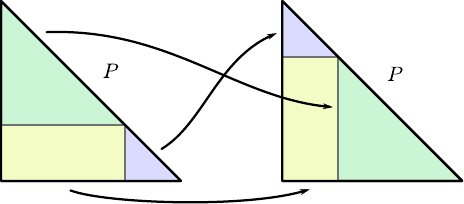}}%
  \end{picture}%
\endgroup%

	}
	\vspace{1em}
	
	We prove that the homology of these groups is independent of the choice of $n$-dimensional polytope, as long as it is nonempty, and that it admits a description in terms of the scissors congruence K-theory spectrum defined by Zakharevich \cite{zakharevich2014}. This spectrum is to the scissors automorphism groups what the algebraic K-theory spectrum of a ring is to the general linear groups. Our work thus leads to a group-theoretic interpretation of Zakharevich's higher scissors congruence K-theory, as well as a method to compute the homology of scissors automorphism groups. This homology computation is often tractable, using the third author's description of the scissors congruence K-theory spectrum in terms of homotopy orbits of a Thom spectrum over a variant of the classical Tits building \cite{malkiewich2022}.
	
	We further extend the results of the paper in two sequels. In \cite{KLMMS-2}, we extend the computational tools developed by the third-named author \cite{malkiewich2022} to any collection of polytopes determined by a fixed set of hyperplanes in the geometry. In \cite{KLMMS-3}, we then calculate the homology of some associated scissors automorphism groups. These include various families of groups of interest in dynamics and group theory, including the rectangular exchange transformations \cite{tanner2023, cornulierlacourte2022} and the Brin--Thompson groups \cite{szymikwahl2019, li2022}. This work therefore provides an interesting connection between these subjects and Zakharevich's K-theory of assemblers \cite{zakharevich2014}.

	We begin by describing our results in detail for Euclidean scissors automorphism groups before discussing the more general context of assemblers, which we use to formulate our results in the body of the paper.

	\subsection{Scissors automorphism groups}
	Two polytopes $P$ and $Q$ in $n$-dimensional Euclidean geometry are \emph{scissors congruent} if $P$ can be cut into finitely many polytopes, and these can be rearranged by isometries to form $Q$. Hilbert's third problem asks for the classification of polytopes up to this relation of scissors congruence. Volume is invariant under scissors congruence, but famously Dehn showed that it is not a complete invariant in dimension 3 \cite{dehn}. Later, Sydler showed that the volume and Dehn invariant suffice in dimension 3 \cite{sydler} and Jessen extended this result to dimension 4 \cite{jessen}. In dimension 5 and above, the problem is still open---see e.g.~\cite{sah1979,dupont_book,zak_perspectives}.

	As mentioned above, in this paper we consider the complementary problem of fixing one polytope $P$, and examining the group of transformations from $P$ to itself that cut $P$ into pieces, and then rearrange those pieces to form $P$. This is the natural extension of the scissors congruence problem, since these transformations are precisely the ones that realize scissors congruences between distinct polytopes $P$ and $Q$.

	\vspace{1em}
	\centerline{
	\def\svgwidth{1.0in}
\begingroup%
  \makeatletter%
  \providecommand\color[2][]{%
    \errmessage{(Inkscape) Color is used for the text in Inkscape, but the package 'color.sty' is not loaded}%
    \renewcommand\color[2][]{}%
  }%
  \providecommand\transparent[1]{%
    \errmessage{(Inkscape) Transparency is used (non-zero) for the text in Inkscape, but the package 'transparent.sty' is not loaded}%
    \renewcommand\transparent[1]{}%
  }%
  \providecommand\rotatebox[2]{#2}%
  \newcommand*\fsize{\dimexpr\f@size pt\relax}%
  \newcommand*\lineheight[1]{\fontsize{\fsize}{#1\fsize}\selectfont}%
  \ifx\svgwidth\undefined%
    \setlength{\unitlength}{65.96903301bp}%
    \ifx\svgscale\undefined%
      \relax%
    \else%
      \setlength{\unitlength}{\unitlength * \real{\svgscale}}%
    \fi%
  \else%
    \setlength{\unitlength}{\svgwidth}%
  \fi%
  \global\let\svgwidth\undefined%
  \global\let\svgscale\undefined%
  \makeatother%
  \begin{picture}(1,1.19959207)%
    \lineheight{1}%
    \setlength\tabcolsep{0pt}%
    \put(0,0){\includegraphics[width=\unitlength,page=1]{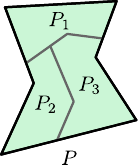}}%
  \end{picture}%
\endgroup%

	}
	
	To make this precise, let $E^n$ denote $n$-dimensional Euclidean geometry, and $I(E^n)$ its \emph{group of isometries}---compositions of translations, rotations, and reflections. A \emph{simplex} $\Delta \subseteq E^n$ is a convex hull of $(n+1)$ points in general position. A \emph{polytope} $P \subset E^n$ is a subset that can be expressed as a finite union of simplices. Note that while a polytope must be bounded, it does not need to be convex, or even connected. A \emph{cover} $\{ P_i \subseteq P \}_{i \in I}$ is a finite set of polytopes $P_i \subseteq P$ whose union is $P$ and whose interiors are disjoint, as illustrated above.

	\begin{definition}A \emph{scissors congruence} $P \congto Q$ is an equivalence class of the data of a cover $\{P_i \subseteq P \}_{i \in I}$ and a collection of isometries $\{ g_i \in I(E^n)\}_{i \in I}$ such that $\{ g_iP_i \subseteq Q \}_{i \in I}$ is a cover, as illustrated below.

	\centerline{
	\def\svgwidth{2.8in}
\begingroup%
  \makeatletter%
  \providecommand\color[2][]{%
    \errmessage{(Inkscape) Color is used for the text in Inkscape, but the package 'color.sty' is not loaded}%
    \renewcommand\color[2][]{}%
  }%
  \providecommand\transparent[1]{%
    \errmessage{(Inkscape) Transparency is used (non-zero) for the text in Inkscape, but the package 'transparent.sty' is not loaded}%
    \renewcommand\transparent[1]{}%
  }%
  \providecommand\rotatebox[2]{#2}%
  \newcommand*\fsize{\dimexpr\f@size pt\relax}%
  \newcommand*\lineheight[1]{\fontsize{\fsize}{#1\fsize}\selectfont}%
  \ifx\svgwidth\undefined%
    \setlength{\unitlength}{174.73782514bp}%
    \ifx\svgscale\undefined%
      \relax%
    \else%
      \setlength{\unitlength}{\unitlength * \real{\svgscale}}%
    \fi%
  \else%
    \setlength{\unitlength}{\svgwidth}%
  \fi%
  \global\let\svgwidth\undefined%
  \global\let\svgscale\undefined%
  \makeatother%
  \begin{picture}(1,0.53060647)%
    \lineheight{1}%
    \setlength\tabcolsep{0pt}%
    \put(0,0){\includegraphics[width=\unitlength,page=1]{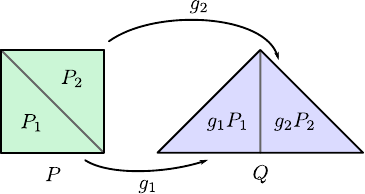}}%
  \end{picture}%
\endgroup%

	}
	
	\noindent The equivalence relation is given by passing to further covers: the pair $\smash{\{P_j' \subseteq P\}_{j \in J}}$ and $\{g_j'\}_{j \in J}$ gives the same scissors congruence if there is a common refinement $\smash{\{P_k'' \subseteq P\}_{k \in K}}$ such that each $P_k''$ is in some intersection $P_i \cap P_j'$ and then we have $g_k'' = g_j' = g_i$.\end{definition}

	\begin{definition}The \emph{scissors automorphism group} $\aut(P)$ is given by scissors congruences from $P$ to itself, under composition.\end{definition}
	
	The reader may find it helpful to think of a scissors automorphism as a function $P \to P$ that is defined almost everywhere, and then composition in the scissors automorphism group is given by composition of such functions. Our first main result concerns the homology of these groups; in its statement, an \emph{abelian local coefficient} system for a group is one where the action factors over its abelianisation.

	\begin{theoremalphabetic}[\cref{thm:embeddings-induce-homology-isomorphisms-vol}]\label{thm:main-iso} For any two nonempty $n$-dimensional polytopes $P$ and $Q$, there is a canonical isomorphism
		\[H_*(\aut(P)) \cong H_*(\aut(Q)).\]
	The same is true with abelian local coefficients.
	\end{theoremalphabetic} 

	These isomorphisms are induced by extension-by-identity homomorphism $\aut(P) \to \aut(Q)$ arising from \emph{scissors embeddings} $P \hookrightarrow Q$, which are defined similarly to scissors congruences by allowing the union of $g_iP_i$ to not be all of $Q$, while retaining the condition that their interiors are disjoint, as illustrated below. We not only prove that the induced map on group homology (possibly with abelian coefficients) is an isomorphism, but also that it does not depend on the choice of scissors embedding. Scissors embeddings $P \hookrightarrow Q$ exist if $\vol(P) < \vol(Q)$ (see \cref{euclidean_ea}), and if $\vol(P) = \vol(Q)$ then the isomorphism in \cref{thm:main-iso} is induced by a zigzag of scissors embeddings of $P$ and $Q$ into any polytope of strictly larger volume. Again, the isomorphism is independent of the choice of zigzag.
	
	\vspace{1em}
	\centerline{
	\def\svgwidth{2.8in}
\begingroup%
  \makeatletter%
  \providecommand\color[2][]{%
    \errmessage{(Inkscape) Color is used for the text in Inkscape, but the package 'color.sty' is not loaded}%
    \renewcommand\color[2][]{}%
  }%
  \providecommand\transparent[1]{%
    \errmessage{(Inkscape) Transparency is used (non-zero) for the text in Inkscape, but the package 'transparent.sty' is not loaded}%
    \renewcommand\transparent[1]{}%
  }%
  \providecommand\rotatebox[2]{#2}%
  \newcommand*\fsize{\dimexpr\f@size pt\relax}%
  \newcommand*\lineheight[1]{\fontsize{\fsize}{#1\fsize}\selectfont}%
  \ifx\svgwidth\undefined%
    \setlength{\unitlength}{180.35181639bp}%
    \ifx\svgscale\undefined%
      \relax%
    \else%
      \setlength{\unitlength}{\unitlength * \real{\svgscale}}%
    \fi%
  \else%
    \setlength{\unitlength}{\svgwidth}%
  \fi%
  \global\let\svgwidth\undefined%
  \global\let\svgscale\undefined%
  \makeatother%
  \begin{picture}(1,0.54709743)%
    \lineheight{1}%
    \setlength\tabcolsep{0pt}%
    \put(0,0){\includegraphics[width=\unitlength,page=1]{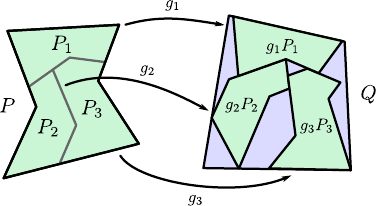}}%
  \end{picture}%
\endgroup%

	}
	\vspace{1em}
	
	Our proof of \cref{thm:main-iso} uses the homological stability machinery of Randal-Williams and Wahl \cite{randalwilliamswahl2017}. As Szymik and Wahl discovered, for ``big'' groups homological stability can present itself as an isomorphism in \emph{all} degrees, as in e.g.~\cite{szymikwahl2019,li2022,PalmerWu}, whose proofs resemble ours. This machinery also gives the following group-theoretic result:

	\begin{corollaryalphabetic}[\cref{thm:quasi-perfectness}]\label{cor:main-quasi-perfect} For any $n$-dimensional polytope $P$, the scissors automorphism group $\aut(P)$ is quasi-perfect, that is, its commutator subgroup is perfect.
	\end{corollaryalphabetic}

	\cref{thm:main-iso} raises the question of what the common value of the homology groups is, in other words, what the \emph{stable homology} is. Our second main result describes this in terms of the infinite loop space of the scissors congruence K-theory spectrum $K(\cE^n)$ defined by Zakharevich \cite{zakharevich2014}.

	\begin{theoremalphabetic}[\cref{cor:stable_homology_acyclic}]\label{thm:main-stable} For any nonempty $n$-dimensional polytope $P$, there is an isomorphism
		\[H_*(\aut(P)) \cong H_*(\Omega^\infty_0 K(\cE^n)).\]
	\end{theoremalphabetic}

	This implies that the homology of $\aut(P)$ is a graded-commutative algebra, which after rationalisation is the free graded-commutative algebra generated by the rational higher scissors congruence $K$-theory groups $K_m(\cE^n) \otimes \Q$, $m \geq 1$.
	
	\smallskip
	
	As part of our proof of \cref{thm:main-stable} we give two alternative descriptions of $\Omega^\infty K(\cE^n)$: one as a group completion (\cref{thm:group-completion}) and one as a plus construction (\cref{thm:plus-construction}),
	\[\Omega^\infty K(\cE^n) \simeq \Omega B(\sqcup_{[P]} B\aut(P)) \simeq K_0(\cE^n) \times B\aut(P)^+.\]
	The coproduct in the middle expression ranges over all scissors congruence classes of $n$-dimensional polytopes, and $P$ in the right-hand expression is any nonempty $n$-dimensional polytope. This is analogous to the construction of classical algebraic K-theory of a ring $R$ via group-completion or the $+$-construction:
	\[ \Omega^\infty K(R) \simeq \Omega B(\sqcup_{[P]} B\mathrm{GL}(P)) \simeq K_0(R) \times B\mathrm{GL}_\infty(R)^+, \]
	where the coproduct in the middle expression ranges over all isomorphism classes of finitely-generated projective $R$-modules (e.g.~\cite[\S IV.1, IV.4.5.1]{weibel2013}). Actually, our description of $\Omega^\infty K(\cE^n)$ is somewhat simpler than in the ring case, because of \cref{thm:main-iso}---it is only necessary to take $B\aut(P)^+$ for a single polytope $P$, rather than taking a colimit over an increasing sequence of polytopes, because all scissors embeddings with domain $P$ induce isomorphisms on homology with abelian local coefficients.
	
	\smallskip

	On the one hand, \cref{thm:main-stable} gives an interpretation of the higher scissors congruence K-theory groups. For example, the Hurewicz theorem identifies the first $K$-group with the abelianisation of $\aut(P)$:
	\[ \aut(P)^{ab} \overset{\cong}\longrightarrow H_1(\aut(P)) \overset{\cong}\longrightarrow \pi_1(K(\cE^n)) = K_1(\cE^n). \]
	We obtain from this a presentation of $K_1(\cE^n)$ with one generator for each scissors automorphism of any polytope $P$ and relations that identify the automorphisms of different polytopes $P$. The corollary below follows from this observation (the authors thank Thor Wittich for suggesting this application):
	
	\begin{corollaryalphabetic}[\cref{complete_presentation}]\label{cor:k1_presentation} The partial presentation of $K_1(\cE^n)$ from \cite[Theorem B]{zak_k1} is in fact a presentation, that is, no additional relations are needed.
	\end{corollaryalphabetic}

	One could imagine a group-theoretic approach to studying higher scissors congruence K-theory groups. For example, in dimension $n \geq 2$ the groups $K_1(\cE^n)$ are unknown. It is a conjecture of Zakharevich that these groups vanish; when translated into a statement about scissors automorphism groups, this becomes the following conjecture:
	
	\begin{conjecturealphabetic}\label{k1_conjecture}
		For any nonempty Euclidean $n$-dimensional polytope $P$ with $n \geq 2$, $\aut(P)$ is perfect. That is, its abelianisation vanishes.
	\end{conjecturealphabetic}
	
	\begin{remark}
		The statement in \cref{k1_conjecture} is true with an elementary proof when $n = 1$. For $n \geq 2$, all known additive invariants---such as the SAF invariant of \cite{saf3,saf1,saf2} and its generalisations in \cite{bgmmz}---have not yet produced a nonzero map from $\aut(P)$ to an abelian group and \cref{k1_conjecture} says that no such map exists.
	\end{remark}

	The discussion above concerns the use of group theory to say new things about higher scissors congruence K-theory. We could also go the other direction, and use tools for scissors congruence K-theory spectra to prove results about scissors automorphism groups that seem difficult to prove directly. For example, by \cite[Theorem 1.6]{malkiewich2022} the groups $K_m(\cE^n)$ are rational when $m \geq 1$ (this is special to Euclidean geometry), and therefore we get as an immediate consequence of \cref{thm:main-stable}:
	\begin{corollaryalphabetic}\label{cor:hautp_formula}
		For any nonempty Euclidean $n$-dimensional polytope $P$, there is an isomorphism
		\[ H_*(\aut(P)) \cong \Lambda^* \left( \bigoplus_{m \geq 1} K_m(\cE^n)[m] \right), \]
		where $\Lambda^*(-)$ denotes the free graded-commutative algebra and $(-)[m]$ denotes that the group is in degree $m$.
	\end{corollaryalphabetic}
	In the sequel \cite{KLMMS-3} we also perform several more explicit calculations. The following is an illustrative example:

	\begin{example}
	For any nonempty $2$-dimensional Euclidean polytope $P$ we have 
		\[H_*(\aut(P)) \cong \Lambda^* \left( \bigoplus_{p+2q \geq 1} H_p(O(2) ; \Lambda^{2q+2}_\Q (\R^2)^t) [p+2q] \right),\]
		where $g \in O(2)$ acts on $\Lambda^{2q+2}_\Q (\R^2)$ via naturality, and the decoration $(-)^t$ adds an additional twist (or determinant representation):
		\[ g \cdot (v_1 \wedge \cdots \wedge v_{2q+2}) = \det(g) (gv_1 \wedge \cdots \wedge gv_{2q+2}). \]
		In particular, the abelianisation of $\aut(P)$ is given by $H_1(O(2);\Lambda^2_\Q (\R^2)^t)$.
	\end{example}

	\subsection{Overview, generalisations and open questions} Our results are not specific to $n$-dimensional polytopes in Euclidean space. An assembler $\cA$ is a category with a reasonable notion of covers and hence of scissors congruence; these were introduced in \cite{zakharevich2014}, and we recall them in \cref{sec:assemblers}. Given an assembler $\cA$, we can define scissors automorphism groups $\aut_\cA(P)$ for objects $P$ of $\cA$ and a scissors K-theory spectrum $K(\cA)$. In \cref{sec:scissors-congruecne-from-a-stability-perspective} and \cref{sec:connection-to-assembler-k-theory}, we show that analogues of our main results (\cref{thm:main-iso}--\cref{thm:main-stable}) hold in this general setting, as long as the assembler satisfies suitable axioms: \emph{EAZ- and S-assemblers} (\cref{def:ass-vol-props} and \cref{def:s-assembler}). We also prove \cref{cor:k1_presentation} for all assemblers. \cref{k1_conjecture} and \cref{cor:hautp_formula}, on the other hand, are special to Euclidean geometry. The main body of the paper is written in this generality, but we invite the reader to keep the example of $n$-dimensional Euclidean polytopes in mind. In \cref{sec:applications-i-classical-sissors-congruence} we show that our EAZ or S axioms hold in the following settings:
	\begin{itemize}
		\item Euclidean, spherical and hyperbolic polytopes.
		\item Smaller transformation groups, e.g.~allowing only translations.
		\item Larger transformation groups, e.g.~allowing scalings.
	\end{itemize}

	The verification of the axioms for $n$-dimensional hyperbolic or spherical polytopes is much more involved than in the Euclidean setting (\cref{sec:ea}). It relies on the following non-trivial geometric input, which we establish along the way and which is of independent interest.
	\begin{theoremalphabetic}[\cref{hyperbolic_ea}, \cref{spherical_ea}]\label{thm:main-ea} For any pair of $n$-dimensional hyperbolic polytopes $P$ and $Q$, if $\vol(P) < \vol(Q)$ then there exists a scissors embedding of $P$ into $Q$. The same holds for $n$-dimensional spherical polytopes as well.
	\end{theoremalphabetic}
	We note that Danny Calegari has also obtained this result in unpublished work, and the authors thank him for sharing his insights.
	
	By varying the allowed polytopes and the transformation groups for e.g.\,$1$-dimensional Euclidean polytopes, we can cover many families of groups in the literature, including the \emph{interval exchange group} $IET$ and \emph{Thompson's group} $V$. In \cite{KLMMS-2,KLMMS-3}, we use these techniques to give an alternative proof of the result of Tanner \cite[Lemma 5.6]{tanner2023} that 
	\[H_*(IET) \cong \Lambda^*\left(\bigoplus_{n \geq 0} \,(\Lambda^{n}_\mathbb{Q} \mathbb{R})[n]\right), \]
	along with a new proof of the result of Szymik--Wahl \cite[Corollary B]{szymikwahl2019} that \[\widetilde{H}_*(V) = 0.\] 
	In fact, it was Tanner's result and its similarity to the calculation in \cite[Theorem 1.11]{malkiewich2022} that inspired the investigations in the present work.
	Our techniques also apply to higher-dimensional analogues of these groups, such as \emph{rectangle exchange groups} \cite{haller_ret,cornulierlacourte2022} and \emph{Brin--Thompson groups} \cite{brin2004, li2022}, which we will explain in \cite{KLMMS-2,KLMMS-3}.
	
	The aforementioned groups are of interest in group theory---Thompson's group $V$ is a famous example of an infinite finitely-presented simple group (see e.g.~\cite{CFP})---and dynamics---interval exchange transformations are prototypical one-dimensional dynamical systems and play a role studying flows on surfaces (see e.g.~\cite{masur_annals,veech_annals}). Li obtained results closely related to ours in the setting of topological full groups arising from generalised dynamical systems known as ample groupoids \cite{li2022}. In an earlier version of this paper, we explained how these fit into the framework of assemblers used in this paper; forthcoming work of Li and the third author will show that the two approaches are equivalent.

	\subsection*{Acknowledgments}

	AK acknowledges the support of the Natural Sciences and Engineering Research Council of Canada (NSERC) [funding reference number 512156 and 512250]. AK is supported by an Alfred P.~Sloan Research Fellowship. CM was partially supported by the National Science Foundation (NSF) grants DMS-2005524, DMS-2052923, and DMS-2506430, a Simons Fellowship, and a Simons Travel Support for Mathematicians grant. He benefited greatly from conversations with Matt Brin, Danny Calegari, Xin Li, John Rached, Daniil Rudenko, and Inna Zakharevich during the writing of this paper. JM was partially supported National Science Foundation (NSF) grants DMS-2202943 and DMS-2504473 as well as a Simons Foundation Travel Support for Mathematicians grant. RJS was supported by NSERC Discovery Grant A4000 and by the German Research Foundation (DFG) through SFB 1442 -- 427320536, Geometry: Deformations and Rigidity, and EXC 2044 -- 390685587, Mathematics M\"unster: Dynamics--Geometry--Structure. RJS would like to thank Xin Li, Michael A.\ Mandell, and Thor Wittich for helpful conversations.
	
	The authors would like to thank the organizers and participants of the Summer school on Scissors Congruence, Algebraic K-theory, and Trace Methods at IU Bloomington in June 2023 for conversations that inspired them to begin working on this paper. We are grateful to Oscar Randal-Williams for pointing out an error in an earlier version. It is a pleasure to thank the anonymous referees for their detailed feedback and careful reading of this work.
	
	\section{Assemblers, scissors automorphisms, and scissors embeddings}\label{sec:assemblers}

	In this section we recall the definition of an assembler $\cA$, which is a category with enough structure to define a scissors congruence relation on its objects. We then define three associated categories:
	\begin{itemize}
		\item the symmetric monoidal category of covers $\cW(\cA)$,
		\item the symmetric monoidal groupoid of scissors congruences $\cG(\cA)$, and
		\item the symmetric monoidal category of scissors embeddings $\category{UG}(\cA)$.
	\end{itemize}
	Most of our work will take place in $\cG(\cA)$, resp.~$\category{UG}(\cA)$, which we think of as the category of formal ``cut-and-paste'' isomorphisms, resp.~embeddings, between formal finite disjoint unions of objects in $\cA$.

	\subsection{Definition of an assembler}\label{sec:def-assemblers}
	This section is based on \cite[Section 2.1]{zakharevich2014}. The starting point of the definition of an assembler is a small \emph{Grothendieck site} $\cA$, i.e.~a small category $\cA$ together with a distinguished collection of \emph{sieves} in the over-category $\cA_{/P}$ of every object $P \in \cA$ that obey a few axioms \cite[Definition 2.3]{zakharevich2014}.
	Given such a site $\cA$, the following notion is used to encode how an object $P$ may be ``assembled'' from the pieces $P_i$: 
	\begin{itemize}
		\item A \emph{covering family} for an object $P$ of $\cA$ is a family of morphisms $\{P_i \to P\}_{i \in I}$ that generates a distinguished sieve over $P$; i.e.~the collection of all morphisms $X \to P$ factoring through some $P_i$ is a distinguished sieve over $P$ in $\cA$.
	\end{itemize}
	Covering families can be composed: if we have a covering family $\{ Q_{ij} \to P_i \}_{j \in J_i}$ for each $P_i$, then the resulting collection of composite maps $\{ Q_{ij} \to P_i \to P \}_{i \in I, j \in J_i}$ forms a covering family of $P$. This allows one to say that one way of assembling $P$ is a refinement of another:
	\begin{itemize}
		\item The composite covering family $\{ Q_{ij} \to P_i \to P \}_{i \in I, j \in J_i}$ is called a \emph{refinement} of the original covering family $\{P_i \to P\}_{i \in I}$.
	\end{itemize}
	We focus on decompositions in which the overlap of the pieces $P_i$ of $P$ is disjoint in the following sense:
	\begin{itemize}
		\item A pair of morphisms $P_1 \to P$, $P_2 \to P$ in $\cA$ is said to be \emph{disjoint} if the pullback $P_1 \times_P P_2$ exists and is initial. A collection of morphisms $\{ P_i \to P\}_{i \in I}$ is disjoint if they are pairwise disjoint.
	\end{itemize} 
	Overall this leads to following class of covering families, which encodes decompositions into finitely many disjoint pieces. These are the ways of ``assembling'' $P$ that we are ultimately interested in.
	
	\begin{definition}\label{def:cover}
		A \emph{cover} of $P \in \cA$ is a disjoint finite covering family of $P$.
	\end{definition}
	
	Zakharevich then defines an assembler as follows \cite[Definition 2.4]{zakharevich2014}:

	\begin{definition}\label{def:assembler} 
		An \emph{assembler} $\cA$ is a small Grothendieck site that satisfies:
		\begin{description}
			\item[\namedlabel{enum:assembler-initial}{(I)}] $\cA$ has an initial object $\varnothing$, and the empty family is a cover of $\varnothing$.
			\item[\namedlabel{enum:assembler-refinement}{(R)}] Every two covers of an object $P$ of $\cA$ have a common refinement.
			\item[\namedlabel{enum:assembler-mono}{(M)}] Every morphism of $\cA$ is a monomorphism.
		\end{description}
	\end{definition}

	\begin{remark}
		\label{rem:distinguished-sieves-and-empty-covers}
		If $S$ is a distinguished sieve over $P$ in $\cA$ and $T$ is any sieve over $P$ containing $S$, it follows from the definition of a Grothendieck topology \cite[Definition 2.3]{zakharevich2014} that $T$ is also a distinguished sieve over $P$. In particular, if the empty sieve, which is generated by the empty family, is distinguished over $P$, then all sieves over $P$ have to be distinguished. Axiom \ref{enum:assembler-initial} asserts this for the initial object $\varnothing$. However, there might be noninitial objects in $\cA$ that also admit an empty cover.
	\end{remark}

	\begin{definition}\label{scissors_congruent}
		Two objects $P$, $Q$ in $\cA$ are \emph{scissors congruent} if there exist covers $\{P_i \to P\}_{i \in I}$ and $\{Q_i \to Q\}_{i \in I}$ with $P_i \cong Q_i$ for every $i \in I$.
	\end{definition}
	
	Informally, two objects are scissors congruent if they can be cut into isomorphic pieces.
	
	\begin{remark}
		Note that any object that admits an empty cover is scissors congruent to the initial object $\varnothing$.
	\end{remark}

	\begin{remark}\label{empty_set_properties}
		By axiom \ref{enum:assembler-mono} and the first part of axiom \ref{enum:assembler-initial} every map $P \to \varnothing$ is an isomorphism and therefore no morphism $P \to \varnothing$ exists if $P \in \cA$ is noninitial \cite[Lemma 2.5]{zakharevich2014}. The second part of axiom \ref{enum:assembler-initial} and \cite[Definition 2.3 (T2)]{zakharevich2014} therefore imply that the empty family is a cover of $P$ if and only if the unique morphism $\{\varnothing \to P\}$ is a cover of $P$. More generally, it follows that the unique morphism $\varnothing \to P$ can be added to or removed from any collection of morphisms $\{ P_i \to P \}_{i \in I}$ without affecting whether that collection is a cover of $P$: If the collection is empty, this was already established. If the collection is nonempty, this holds because we can only map from initial objects to $\varnothing$ and these are contained in any nonempty sieve over $P$.
	\end{remark}

	\subsection{The category of covers \texorpdfstring{$\cW(\cA)$}{W(A)}}\label{sec:category-covers}
	The study of scissors congruence in an assembler $\cA$ begins by defining a second category $\cW(\cA)$, in which an object is a formal finite disjoint union of objects in $\cA$, and a morphism is a finite collection of covers in $\cA$.

	\begin{definition}\label{def:ccategory-covers}
		The \emph{category of covers} $\cW(\cA)$ of an assembler $\cA$ is given as follows:
		\begin{itemize}
			\item The objects are pairs $\scr{P} = (I,\{P_i\}_{i \in I})$ of a finite set $I$ and a collection of objects $P_i$ of $\cA$ for $i \in I$.
			\item The morphisms $\scr{P} = (I,\{P_i\}_{i \in I}) \to \scr{Q} = (J,\{Q_j\}_{j \in J})$ are pairs $(\phi,\{f(i)\}_{i \in I})$ of a map of finite sets $\phi \colon I \to J$ and a collection of morphisms $f(i) \colon P_i \to Q_{\phi(i)}$ in $\cA$ for $i \in I$, such that $\{f(i)\colon P_i \to Q_j \}_{i \in \phi^{-1}(j)}$ is a cover of $Q_j$, for each $j \in J$.
			\item The composition is given by
			\[(\psi,\{g(j)\})_{j \in J}) \circ (\phi,\{f(i)\}_{i \in I}) \coloneqq \big(\psi \circ \phi,\{g(\phi(i)) \circ f(i)\}_{i \in I}\big).\]
		\end{itemize}	
		A morphism in $\cW(\cA)$ from a 5-tuple $\scr{P} = \{P_1, P_2, P_3, P_4, P_5\}$ to a 2-tuple $\scr{Q} = \{Q_1, Q_2\}$ is pictured below.
		
			\centerline{
	\def\svgwidth{4.3in}
\begingroup%
  \makeatletter%
  \providecommand\color[2][]{%
    \errmessage{(Inkscape) Color is used for the text in Inkscape, but the package 'color.sty' is not loaded}%
    \renewcommand\color[2][]{}%
  }%
  \providecommand\transparent[1]{%
    \errmessage{(Inkscape) Transparency is used (non-zero) for the text in Inkscape, but the package 'transparent.sty' is not loaded}%
    \renewcommand\transparent[1]{}%
  }%
  \providecommand\rotatebox[2]{#2}%
  \newcommand*\fsize{\dimexpr\f@size pt\relax}%
  \newcommand*\lineheight[1]{\fontsize{\fsize}{#1\fsize}\selectfont}%
  \ifx\svgwidth\undefined%
    \setlength{\unitlength}{284.98535803bp}%
    \ifx\svgscale\undefined%
      \relax%
    \else%
      \setlength{\unitlength}{\unitlength * \real{\svgscale}}%
    \fi%
  \else%
    \setlength{\unitlength}{\svgwidth}%
  \fi%
  \global\let\svgwidth\undefined%
  \global\let\svgscale\undefined%
  \makeatother%
  \begin{picture}(1,0.5712364)%
    \lineheight{1}%
    \setlength\tabcolsep{0pt}%
    \put(0,0){\includegraphics[width=\unitlength,page=1]{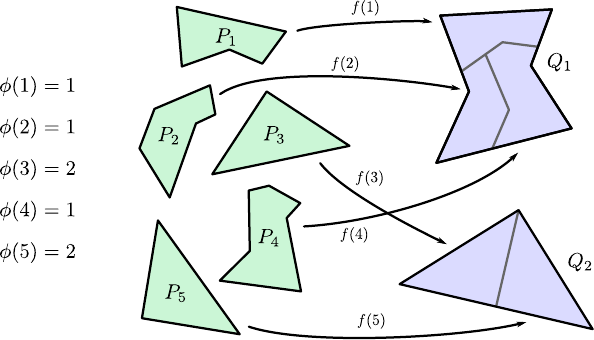}}%
  \end{picture}%
\endgroup%

	}
	\vspace{1em}
	
		The category $\cW(\cA)$ admits a symmetric monoidal structure whose tensor product is given by taking disjoint unions
		\[\scr{P} \sqcup \scr{Q} = (I,\{P_i\}_{i \in I}) \sqcup (J,\{Q_j\}_{j \in J}) \coloneqq (I \sqcup J,\{P_i\}_{i \in J} \sqcup \{Q_j\}_{j \in J})\]
		on objects and analogously on morphisms. The unit object in $\cW(\cA)$ is the pair $\varnothing = (\varnothing, \varnothing)$ given by the empty set and the empty collection of objects in $\cA$.
	\end{definition}

	Observe that two singletons $\{P\}$ and $\{Q\}$ are connected by a zigzag of morphisms in $\cW(\cA)$ if and only if $P$ and $Q$ are scissors congruent. In this sense $\cW(\cA)$ encodes the relation of scissors congruence, though it has the disadvantage that the cuttings and pastings point in opposite directions and can not be composed. \cref{scissors_congruence_groupoid} will address this (see also \cref{rem:maps-into-and-out-of-varnothing}).
	
	\begin{notation}
		By abuse of notation, we will call a morphism $\scr{P} \to \scr{Q}$ in $\cW(\cA)$ a \emph{cover}, because it is nothing more than a formal finite disjoint union of covers in $\cA$.
	\end{notation}

	We now make some remarks about this definition:

	\begin{remark}Despite the notation, the tensor product on $\cW(\cA)$ does not in general define a coproduct in the category $\cW(\cA)$.
	\end{remark}

	\begin{remark}
		\label{rem:maps-into-and-out-of-varnothing}
		On the one hand, there is a morphism $\varnothing = (\varnothing, \varnothing) \to \scr{P} = (I, \{P_i\})$ in $\cW(\cA)$ if and only if each $P_i$ has an empty cover. On the other hand, there is a morphism $\scr{P} = (I, \{P_i\}) \to \varnothing = (\varnothing, \varnothing)$ in $\cW(\cA)$ if and only if $\scr{P} = \varnothing$.
	\end{remark}

	\begin{remark}\label{empty_sets_in_W_discussion}
		Our definition of $\cW(\cA)$ differs slightly from that in \cite[Section 2.1]{zakharevich2014} in that we allow the initial objects to occur in the tuple of objects $\scr{P} = (I,\{P_i\}_{i \in I})$. These two variants of $\cW(\cA)$ have equivalent classifying spaces: letting $\cW^\circ(\cA)$ be the variant without initial objects, the evident inclusion $i^\circ \colon \cW^\circ(\cA) \to \cW(\cA)$ admits a right adjoint $p^\circ$ that deletes all the instances of initial objects from the tuples. To see this is well-defined we need \cref{empty_set_properties}---deleting instances of initial objects does not affect whether something is a cover---and axiom \ref{enum:assembler-initial}---an initial object has an empty cover. The lack of initial objects $\varnothing$ in the tuples was assumed in \cite[Section 2.1]{zakharevich2014} in order to show that $\cW(\cA)$ consists of monomorphisms.
	\end{remark}

	\subsection{The category of scissors congruences \texorpdfstring{$\cG(\cA)$}{G(A)}} \label{sec:scissors-congruences}
	Scissors congruences will be equivalence classes of zigzags of morphisms in $\cW(\cA)$, and we can conveniently express these as morphisms in a category of fractions. Recall that for a category $\catC$ with a wide subcategory $W$ of weak equivalences, the localisation $\catC[W^{-1}]$ is modelled by a category of fractions provided that the following conditions are satisfied \cite[Chapter 2]{gabrielzisman1967}:

	\begin{definition}\label{def:loc-conditions} Let $(\catC,W)$ be a pair of a category and a wide subcategory.
	\begin{enumerate}
		\item \label{enum:ore} $(\catC,W)$ satisfies the \emph{Ore condition} if given $f \colon x \rightarrow z$ in $W$ and $g\colon y \rightarrow z$ in $\catC$ there exists $h \colon w \rightarrow x $ in $\catC$ and $i \colon w \rightarrow y $ in $W$ such that the following square commutes:
		\[ \begin{tikzcd}
			w \ar[r, "h"] \ar[d, "i"] & x \ar[d, "f"]\\
			y \ar[r,"g"] & z.
		\end{tikzcd} \]
		\item \label{enum:equalising} $(\catC,W)$ satisfies the \emph{equalising condition} if given a diagram
		\[ \begin{tikzcd} x \ar[r, "f", shift left=.75ex] \ar[r, "g"', shift right = .75ex] & y \ar[r, "h"] & z \end{tikzcd}\]
		with $h$ in $W$ and such that $hf=hg$, there exists a morphism $i :w \rightarrow x$ in $W$ such that $fi=gi$.
	\end{enumerate}
	\end{definition}

	\begin{definition}\label{def:loc} The \emph{category of fractions} $\catC[W^{-1}]$ of a pair $(\catC,W)$ satisfying \cref{def:loc-conditions} \ref{enum:ore} and \ref{enum:equalising}, is given as follows:
	\begin{itemize}
		\item The objects are those of $\catC$.
		\item The morphisms are equivalence classes of spans with the left leg in $W$:
		\[ \catC[W^{-1}](a,b) = \left\{ a \xleftarrow{f} a' \xrightarrow{g} b \mid f \in W \right\} \big /{\sim}\]
		where $(a \xleftarrow{f} a' \xrightarrow{g} b) \sim (a \xleftarrow{h} a'' \xrightarrow{i} b)$ if there exists an object $x$ and morphisms $u,v$ such that the following diagram commutes:
		\[\begin{tikzcd}[sep=small]
			& a' \ar[dl , "f"'] \ar[dr, "g"]& \\
			a & \ar[u, "u"']x \ar[d, "v"]& b \\
			& a'' \ar[ul, "h"]\ar[ur, "i"']&
		\end{tikzcd}\]
		and $fu = hv$ are morphisms in $W$.

		\item The composition of spans $b \leftarrow b' \rightarrow c$ and $a \leftarrow a' \rightarrow b$ is given by applying the Ore condition to construct the diagram
		\[ \begin{tikzcd}[sep=small]
			a  &a' \ar[l] \ar[r] &b \\
			& z \ar[u] \ar[r] & b' \ar[u] \ar[d]\\
			& & c,
		\end{tikzcd}\]
		and then we set $(b \leftarrow b' \rightarrow c)\circ (a \leftarrow a' \rightarrow b ) = (a \leftarrow a' \leftarrow z \rightarrow b' \rightarrow c )$.
		\end{itemize}
	There is a localisation functor $\catC \rightarrow \catC[W^{-1}]$, given by the identity on objects and on morphisms by 
	\[(c \xrightarrow{f} c') \longmapsto (c \xleftarrow{1_c} c \xrightarrow{f} c').\]
	\end{definition}

	This construction is compatible with symmetric monoidal structures \cite{Day}:

	\begin{lemma}[Day] \label{lem:loc-sym} Suppose we have a pair $(\catC,W)$ satisfying \cref{def:loc-conditions} \ref{enum:ore} and \ref{enum:equalising}. If $\catC$ is symmetric monoidal and $W$ is closed under the tensor product, then pointwise tensor product of spans induces a symmetric monoidal structure on $\catC[W^{-1}]$ so that the localisation functor $\iota \colon \catC \to \catC[W^{-1}]$ is symmetric monoidal.
	\end{lemma}

	\begin{lemma}\label{lem:covers-loc} 
	If $\cA$ is an assembler then the pair $(\cW(\cA), \cW(\cA))$ satisfies the conditions in \cref{def:loc-conditions} \ref{enum:ore} and \ref{enum:equalising}.
	\end{lemma}

	\begin{proof}
	Recall the category $\cW^\circ(\cA)$ and the adjunction $i^\circ \dashv p^\circ$ from \cref{empty_sets_in_W_discussion}.
	\cite[Proposition 2.11]{zakharevich2014} says that in $\cW^\circ(\cA)$ all cospans can be completed to commutative squares and all morphisms are monomorphisms, and hence $(\cW^\circ(\cA), \cW^\circ (\cA))$ satisfies \cref{def:loc-conditions} \ref{enum:ore} and \ref{enum:equalising}.  We will use this and $i^\circ \dashv p^\circ$ to prove the corresponding result for $\cW(\cA)$.
	
	First we verify \cref{def:loc-conditions} \ref{enum:equalising} for $\cW(\cA)$. We start with a diagram
	\[ \begin{tikzcd} \scr{P} \ar[r, "f", shift left=.75ex] \ar[r, "g"', shift right = .75ex] & \scr{Q} \ar[r, "h"] & \scr{R} \end{tikzcd}\]
	in $\cW(\cA)$ with $hf=hg$. We can construct a larger commuting diagram with the bottom row in $\cW^\circ (\cA)$ using the adjunction $i^\circ \dashv p^\circ$, where we write $(-)^\circ \coloneqq i^\circ p^\circ(-)$ and $\epsilon \colon i^\circ p^\circ \Rightarrow \id$ for the counit:
	\[ \begin{tikzcd}
	\scr{P} \ar[r, "f", shift left=.75ex] \ar[r, "g"', shift right = .75ex]  & \scr{Q} \ar[r, "h"] & \scr{R} \\
	\scr{P}^\circ \ar[u, "\epsilon_\scr{P}"]\ar[r, "f^{\circ}", shift left=.75ex] \ar[r, "g^{\circ}"', shift right = .75ex]  &\scr{Q}^\circ \ar[u, "\epsilon_{\scr{Q}}"']\ar[r, "h^\circ"] & \scr{R}^\circ \ar[u, "\epsilon_\scr{R}"]
	\end{tikzcd}\]
	Since all morphisms in $\cW^\circ (\cA)$ are monomorphisms, we have $f^\circ=g^\circ$. Using commutativity, we have  $f \epsilon_\scr{P}= \epsilon_\scr{Q} f^\circ = \epsilon_\scr{Q} g^\circ = g \epsilon_\scr{P}$.

	Next we verify \cref{def:loc-conditions} \ref{enum:ore} for $\cW(\cA)$. Given a cospan $\smash{\scr{P} \xrightarrow{\phi} \scr{R} \xleftarrow{\psi} \scr{Q}}$ in $\cW(\cA)$, by \cref{def:assembler} \ref{enum:assembler-refinement} applied to each object in the tuple $\scr{R}$, we form a common refinement tuple $\scr{L}$, which fits into a commuting square
	
	\[\begin{tikzcd}
	\scr{L} \dar \rar & \scr{Q} \ar[d, "\psi"] \\
	\scr{P} \ar[r,"\phi"'] & \scr{R}.
	\end{tikzcd}\]
	\end{proof}

	\begin{definition}\label{scissors_congruence_groupoid}
		The \emph{scissors congruence groupoid} $\cG(\cA)$ is $\cW(\cA)[\cW(\cA)^{-1}]$.
	\end{definition}

	We will refer to morphisms $\scr{P} \congto \scr{Q}$ in $\cG(\cA)$ as \emph{scissors congruences}. We say that two finite collections of objects $\scr{P}$ and $\scr{Q}$ in $\cA$ are \emph{scissors congruent}, denoted $\scr{P} \simeq \scr{Q}$, if there is a scissors congruence between them---note that this generalizes \cref{scissors_congruent} from individual objects in $\cA$ to finite tuples of objects. Concretely, a scissors congruence $f \colon \scr{P} \congto \scr{Q}$ is represented by a zig-zag of covers
	\begin{equation}\label{map_in_G}
	\scr{P} = (I,\{P_i\}_{i \in I}) \longleftarrow \scr{R} = (L,\{R_\ell\}_{\ell \in L}) \longrightarrow \scr{Q} = (J,\{Q_j\}_{j \in J}).
	\end{equation}
	That is, it is given by a pair of maps of sets $\phi \colon L \to I$ and $\psi \colon L \to J$, together with morphisms in $\cA$ that make $\{R_\ell\}_{\ell \in \phi^{-1}(i)}$ into a cover of $P_i$ for each $i \in I$, and $\{R_\ell\}_{\ell \in \psi^{-1}(j)}$ into a cover of $Q_j$ for each $j \in J$. Two such zig-zags represent the same morphism in $\cG(\cA)$ if they become the same after passing to a refinement of the middle term.
	
	\vspace{1em}
			\centerline{
	\def\svgwidth{5.2in}
\begingroup%
  \makeatletter%
  \providecommand\color[2][]{%
    \errmessage{(Inkscape) Color is used for the text in Inkscape, but the package 'color.sty' is not loaded}%
    \renewcommand\color[2][]{}%
  }%
  \providecommand\transparent[1]{%
    \errmessage{(Inkscape) Transparency is used (non-zero) for the text in Inkscape, but the package 'transparent.sty' is not loaded}%
    \renewcommand\transparent[1]{}%
  }%
  \providecommand\rotatebox[2]{#2}%
  \newcommand*\fsize{\dimexpr\f@size pt\relax}%
  \newcommand*\lineheight[1]{\fontsize{\fsize}{#1\fsize}\selectfont}%
  \ifx\svgwidth\undefined%
    \setlength{\unitlength}{351.75189941bp}%
    \ifx\svgscale\undefined%
      \relax%
    \else%
      \setlength{\unitlength}{\unitlength * \real{\svgscale}}%
    \fi%
  \else%
    \setlength{\unitlength}{\svgwidth}%
  \fi%
  \global\let\svgwidth\undefined%
  \global\let\svgscale\undefined%
  \makeatother%
  \begin{picture}(1,0.27027988)%
    \lineheight{1}%
    \setlength\tabcolsep{0pt}%
    \put(0,0){\includegraphics[width=\unitlength,page=1]{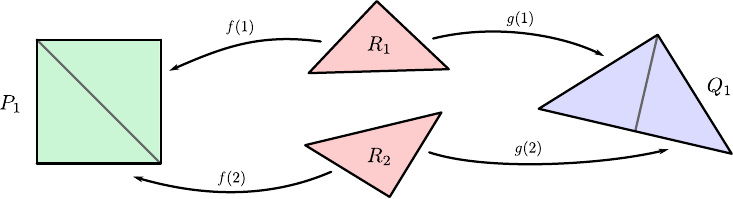}}%
  \end{picture}%
\endgroup%

	}
	\vspace{1em}
	
	\begin{remark}
	Any two covers of $\scr{P}$ have a common refinement, so two zig-zags as above can be refined so as to give the same cover of $\scr{P}$. Similarly, we could refine them so to give the same cover of $\scr{Q}$. But we cannot do both of these things simultaneously, unless the zig-zags define the same morphism in $\cG(\cA)$.
	\end{remark}
	
	\begin{definition}\label{scissors_automorphism_group}
	For any $\scr{P} \in \ob(\cG(\cA))$, the \emph{scissors automorphism group} $\aut(\scr{P})$ is its automorphism group in $\cG(\cA)$.
	\end{definition}
	
	\begin{remark}
	It would be convenient to call $\aut(\scr{P})$ the ``scissors congruence group,'' but that term is already reserved for the group completion of $\pi_0\,\cG(\cA)$. The term ``automorphism'' emphasizes that we are considering cut-and-paste maps from a single object to itself.
	\end{remark}

	By \cref{lem:loc-sym}, $\cG(\cA)$ admits a canonical structure of a symmetric monoidal groupoid. We denote its tensor product by $\sqcup$ and we let $\varnothing$ denote to the empty tuple in $\cG(\cA)$, which is the unit of the symmetric monoidal structure. \cref{rem:maps-into-and-out-of-varnothing} implies the following characterisation of the unit up to isomorphism in $\cG(\cA)$:

	\begin{lemma}\label{sc_to_empty}
		$\scr{P} \simeq \varnothing$ in $\cG(\cA)$ if and only if every $P_i \in \scr{P}$ has an empty cover.
	\end{lemma}
	
	\begin{definition}\label{empty_meaning}
	We say that an object $\scr{P}$ of $\cG(\cA)$ is \emph{empty} when it satisfies the equivalent conditions of \cref{sc_to_empty}.
	\end{definition}
	
	\subsection{The category of scissors embeddings \texorpdfstring{$\category{UG}(\cA)$}{UG(A)}}

	Following the work of Randal-Williams and Wahl, in particular \cite[Section 1.1]{randalwilliamswahl2017}, we can associate  to $\cG(\cA)$ another symmetric monoidal category with the same objects, but whose morphisms are ``embeddings'' determined by the symmetric monoidal structure. This construction goes back to Quillen \cite{grayson_higher2}.

	\begin{definition}\label{scissors_embedding_category}
	The category $\category{UG}(\cA)$ is given as follows:
	\begin{itemize}
		\item The objects are those of $\cG(\cA)$.
		\item The morphisms $e\colon \scr{P} \to \scr{Q}$ in $\category{UG}(\cA)$ are equivalence classes of scissors congruences $\scr{P} \sqcup \scr{P'} \congto \scr{Q}$ in $\cG(\cA)$, where $\scr{P'}$ is an object in $\cG(\cA)$. Here two such scissors congruences are equivalent if there exists a commutative diagram in $\cG(\cA)$
		\[\begin{tikzcd}[sep=small] \scr{P} \sqcup \scr{P'} \arrow{rd}{\simeq} \arrow{dd}{\simeq}[swap]{\id_\scr{P} \sqcup \phi} &[20pt] \\[-5pt]
			& \scr{Q}. \\[-5pt]
			\scr{P} \sqcup \scr{P''}  \arrow{ru}[swap]{\simeq} &\end{tikzcd}\] 
		\item The composition is induced by the composition in $\cG(\cA)$.
	\end{itemize}  
	We shall refer to the morphisms in $\category{UG}(\cA)$ as \emph{scissors embeddings}.
	\end{definition}

	\begin{definition}\label{def:complement}
	A \emph{complement} of a scissors embedding $e \colon  \scr{P} \to \scr{Q}$ is any $\scr{P'} \in \cG(A)$ that can be used to represent $e$ as a scissors congruence $\smash{\scr{P} \sqcup \scr{P'} \congto \scr{Q}}$. It is evidently well-defined up to isomorphism in $\cG(\cA)$. We sometimes use the notation $\scr{Q} - e(\scr{P})$ to refer to any choice of complement of $e$.\end{definition}

	Using that $\scr{P'} = \scr{Q} - e(\scr{P})$ is unique up to isomorphism in $\cG(\cA)$, we could have also defined the morphisms in $\category{UG}(\cA)$ by
	\[\hom_{\category{UG}(\cA)}(\scr{P},\scr{Q}) = \left(\bigsqcup_{\scr{P}'} \hom_{\cG(\cA)}(\scr{P} \sqcup \scr{P}',\scr{Q})\right)/\aut_{\cG(\cA)}(\scr{P}'),\]
	where the disjoint union runs over all isomorphism classes of objects in $\cG(\cA)$, and where $f \in \aut_{\cG(\cA)}(\scr{P}')$ acts by precomposition with $\id_\scr{P} \sqcup f$.

	Any zigzag $\scr{P} \sqcup \scr{P'} \leftarrow \scr{R} \rightarrow \scr{Q}$ representing a scissors embedding $e \colon \scr{P} \hookrightarrow \scr{Q}$ yields three morphisms $\scr{R}_{\scr{P}} \to \scr{P}$, $\scr{R}_{\scr{P'}} \to \scr{P'}$, and $\scr{R} = \scr{R}_{\scr{P}} \sqcup \scr{R}_{\scr{P'}} \to \scr{Q}$ in $\cW(\cA)$:
	\begin{align*}
		\scr{P} = (I,\{P_i\}_{i \in I}) \longleftarrow \scr{R}_{\scr{P}} & = (L_{\scr{P}},\{R_\ell\}_{\ell \in L_{\scr{P}}}), \\
		\scr{P}' = (I',\{P'_{i'}\}_{i' \in I'}) \longleftarrow \scr{R}_{\scr{P}'} & = (L_{\scr{P}'},\{R_{\ell}\}_{\ell \in L_{\scr{P}'}}), \\
		\scr{R}_{\scr{P}} \sqcup \scr{R}_{\scr{P'}} & = (L_{\scr{P}} \sqcup L_{\scr{P'}},\{R_\ell\}_{\ell \in L_{\scr{P}} \sqcup L_{\scr{P'}}}) \longrightarrow \scr{Q} = (J,\{Q_j\}_{j \in J})
	\end{align*}
	 The next lemma shows that we can forget the part of the data corresponding to $\scr{R}_{\scr{P'}}$.

	\begin{lemma}\label{sc_emb_simplified}
		A scissors embedding is uniquely determined by the following data:
		\begin{enumerate}
			\item $\scr{R}_{\scr{P}}$, 
			\item the functions $\phi\colon L_{\scr{P}} \to I$ and $\psi\colon L_{\scr{P}} \to J$, and 
			\item the maps $R_\ell \to P_{\phi(\ell)}$ and $R_\ell \to Q_{\psi(\ell)}$.
		\end{enumerate}
		Such data defines a scissors embedding if and only if the resulting map $\scr{R}_{\scr{P}} \to \scr{P}$ is a cover (i.e.~a map in $\cW(\cA)$), and the map $\scr{R}_{\scr{P}} \to \scr{Q}$ extends to a cover. Two collections of such data give the same scissors embedding iff they agree up to refinement.
	\end{lemma}

	\begin{proof}
	If the maps extend to a cover of $\scr{Q}$, the complement is unique up to scissors congruence. We see this by taking any two such complements $\scr{R}_1$ and $\scr{R}_2$, taking a common refinement of the resulting covers of $\scr{Q}$, and observing that such a refinement induces a common refinement of $\scr{R}_1$ and $\scr{R}_2$.
	\end{proof}

	Since a scissors embedding is determined by a cover $\scr{R}_{\scr{P}} \to \scr{P}$ and a partial cover $\scr{R}_{\scr{P}} \to \scr{Q}$, it makes sense to speak of scissors embeddings with the same target being disjoint, or forming a cover: we simply consider the partial covers $\scr{R}_{\scr{P}} \to \scr{Q}$ and ask whether the resulting collections of morphisms to each $Q_j$ are disjoint, or cover.

	\begin{definition}\label{df:disjoint}
		A finite collection of scissors embeddings $e_i\colon \scr{P}_i \to \scr{Q}$ is \emph{disjoint} if there exist spans representing the embeddings in which for each $j \in J$, the corresponding maps $R_\ell \to Q_j$ for $\ell \in L_{\scr{P}_i}$ (not $L_{\scr{P}_i'}$) are disjoint. The embeddings $e_i$ are a \emph{cover} if they are disjoint in this sense, and in addition the corresponding maps $R_\ell \to Q_j$ for $\ell \in L_{\scr{P}_i}$ form a cover of $Q_j$.
	\end{definition}

	\begin{remark}
		If $\cA$ has objects with empty covers, then it is possible to add such objects to a disjoint pair of collections of maps $R_\ell \to Q$, and get a new pair of collections of maps that are non-disjoint. This is why the definition asks for \emph{some} span in which the maps are disjoint, not that \emph{all} spans have this property.
	\end{remark}

	\begin{example}
		\label{example:cover-of-scissors-embeddings}
		For any finite collection $\{\scr{P}_i\}_{i=1}^n$ of objects in $\cG(\cA)$, there are canonical scissors embeddings $\iota_{\scr{P}_i} \colon \scr{P}_i \hookrightarrow \sqcup_{i=1}^n \scr{P}_i$ and together these cover.\end{example}

	\begin{remark}
		In most examples, there is a functor $\cA \to \category{UG}(\cA)$, but this is not true in general. It amounts to asking that every morphism in $\cA$ can be completed to a cover.
	\end{remark}

	\subsubsection{Properties of scissors embeddings} We now establish some properties of scissors embeddings.

	\begin{lemma}\label{disjoint_iff_complement}
		Scissors embeddings $e_1\colon \scr{P}_1 \to \scr{Q}$ and $e_2\colon \scr{P}_2 \to \scr{Q}$ are disjoint if and only if there is a factorisation of $e_2$ through the complement $\scr{Q} - e_1(\scr{P}_1)$.
	\end{lemma}

	\begin{proof}
		Let $\scr{P}_1'$ denote the complement $\scr{Q} - e_1(\scr{P}_1)$ and suppose that a factorisation $\scr{P}_2 \to \scr{P}_1' \to \scr{Q}$ of $e_2$ exists.
		The two scissors embeddings in this factorisation are represented by covers of the form
		\[ \scr{R}_{\scr{P}_1} \to \scr{P}_1, \quad \scr{R}_{\scr{P}_1'} \to \scr{P}_1', \quad \scr{R}_{\scr{P}_1} \sqcup \scr{R}_{\scr{P}_1'} \to \scr{Q}, \]
		\[ \scr{R}_{\scr{P}_2} \to \scr{P}_2, \quad \scr{R}_{\scr{P}_2'} \to \scr{P}_2', \quad \scr{R}_{\scr{P}_2} \sqcup \scr{R}_{\scr{P}_2'} \to \scr{P}_1'. \]
		This involves two covers of $\scr{P}_1'$, namely $\scr{R}_{\scr{P}_2} \sqcup \scr{R}_{\scr{P}_2'} \to \scr{P}_1'$ and $\scr{R}_{\scr{P}_1'} \to \scr{P}_1'$. To calculate the composition of the two scissors embeddings, we pass to a refinement $\scr{R}_2$ of $\scr{R}_{\scr{P}_2}$ and a refinement $\scr{R}_3$ of $\scr{R}_{\scr{P}_2'}$, such that the union $\scr{R}_2 \sqcup \scr{R}_3$ is also a refinement of $\scr{R}_{\scr{P}_1'}$. Simplifying the notation by writing $\scr{R}_1$ for $\scr{R}_{\scr{P}_1}$, this gives a new set of covers of the form
		\[ \scr{R}_1 \to \scr{P}_1, \quad \scr{R}_2 \sqcup \scr{R}_3 \to \scr{P}_1', \quad \scr{R}_1 \sqcup \scr{R}_2 \sqcup \scr{R}_3 \to \scr{Q}, \]
		\[ \scr{R}_2 \to \scr{P}_2, \quad \scr{R}_3 \to \scr{P}_2', \quad \scr{R}_2 \sqcup \scr{R}_3 \to \scr{P}_1'. \]
		Since passing to a refinement does not affect the resulting morphism in $\category{UG}(\cA)$, we conclude that the cover $\scr{R}_1 \to \scr{P}_1$ and the partial cover $\scr{R}_1 \to \scr{Q}$ represent $e_1$, while the cover $\scr{R}_2 \to \scr{P}_2$ and the partial cover $\scr{R}_2 \to \scr{Q}$ represent $e_2$. Since $\scr{R}_1 \sqcup \scr{R}_2 \to \scr{Q}$ is a partial cover, each pair of maps $R_\ell, R_{\ell'} \rightrightarrows Q_j$ inside this family is disjoint, so by definition $e_1$ and $e_2$ are disjoint.

		Conversely, suppose the embeddings $e_1\colon \scr{P}_1 \to \scr{Q}$ and $e_2\colon \scr{P}_2 \to \scr{Q}$ are disjoint. Then we may represent them by covers and partial covers
		\[ \scr{R}_1 \to \scr{P}_1, \quad \scr{R}_1 \to \scr{Q}, \qquad \scr{R}_2 \to \scr{P}_2, \quad \scr{R}_2 \to \scr{Q}\]
		such that every pair of maps in $\scr{R}_1 \to \scr{Q}$ and $\scr{R}_2 \to \scr{Q}$ with the same target $Q_j$ are disjoint. Taking a common subdivision as covers of $\scr{Q}$ and eliminating any $\varnothing$ terms, we get covers of the form
		\[ \scr{R}_1 \to \scr{P}_1, \quad \scr{R}_2 \to \scr{P}_2, \quad \scr{R} \to \scr{Q} \]
		such that $\scr{R}_1$ and $\scr{R}_2$ are subsets of $\scr{R}$. We need to verify that $\scr{R}_1$ and $\scr{R}_2$ are disjoint subsets of $\scr{R}$. Let us examine any one object $R_\ell \to Q_j$ that appears in both subsets. It occurs as a refinement of two objects $R_1 \to Q_j$ and $R_2 \to Q_j$ in the two sets $\scr{R}_1$ and $\scr{R}_2$, such that $R_1 \times_{Q_j} R_2 \cong \varnothing$. This implies the existence of a map $R_\ell \to \varnothing$, which by \cref{empty_set_properties} means that $R_\ell \cong \varnothing$. But we assumed there were no such terms. Therefore $\scr{R}_1$ and $\scr{R}_2$ are disjoint subsets of $\scr{R}$, giving a factorisation of $e_2$ through $\scr{Q} - e_1(\scr{P}_1)$.
	\end{proof}

	\begin{lemma}\label{disjoint_implies_conjugate}
		If $e,e'\colon \scr{P} \rightrightarrows \scr{Q}$ are disjoint embeddings of the same object $\scr{P}$, then there exists a scissors congruence $h\colon \scr{Q} \to \scr{Q}$ such that $he_1 = e_2$.
	\end{lemma}

	\begin{proof}
		Taking common refinements and eliminating any objects with empty covers, we get one cover of $\scr{Q}$ with two disjoint sub-tuples, each of which also forms a cover of $\scr{P}$. Taking common refinements of those two covers of $\scr{P}$, we get a single cover $\scr{R} \to \scr{P}$, and now our cover of $\scr{Q}$ is of the form $\scr{R} \sqcup \scr{R} \sqcup \scr{Q'}$. Now define $h$ to be the scissors congruence that interchanges the two copies of $\scr{R}$.
	\end{proof}

	\begin{definition}
		\label{def:intersection-and-union}
		Given a finite collection of scissors embeddings $e_i\colon \scr{P}_i \to \scr{Q}$, the \emph{intersection} $\bigcap_i e(\scr{P}_i)$ is defined by taking a common refinement of all the covers $\scr{R}_{\scr{P}_i} \sqcup \scr{R}_{\scr{P}'_i} \to \scr{Q}$, then selecting those pieces lying in $\scr{R}_{\scr{P}_i}$ for every value of $i$. The \emph{union} $\bigcup_i e(\scr{P}_i)$ is similarly defined by taking a common refinement and then taking those pieces lying in $\scr{R}_{\scr{P}_i}$ for at least one value of $i$. By construction, the intersection and the union both come with preferred scissors embeddings into $\scr{Q}$.
	\end{definition}
	
	\begin{remark}
		\label{rem:disjoint-scissors-embeddings}
		Scissors embeddings are disjoint if and only if $\bigcap_i e_i(\scr{P}_i) \simeq \varnothing$. The forward direction is obvious; for the backward direction, if $\bigcap_i e_i(\scr{P}_i) \simeq \varnothing$ then by \cref{sc_to_empty} the intersection consists entirely of objects with empty covers. Up to refinement we can remove all such objects, giving a presentation of the embeddings in which their intersection is the empty tuple and therefore the embeddings are disjoint.
	\end{remark}
	
	\begin{corollary}\label{make_embeddings_overlap}
		For any two scissors embeddings $f,g\colon \scr{P} \rightrightarrows \scr{Q}$ with $\scr{P} \not\simeq \varnothing$, there exists a scissors automorphism $h\colon \scr{Q} \to \scr{Q}$ such that $hf(\scr{P}) \cap g(\scr{P}) \not\simeq \varnothing$.
	\end{corollary}
	\begin{proof}
		Either $f(\scr{P}) \cap g(\scr{P}) \not\simeq \varnothing$, so $h = \id$ works, or else $f$ and $g$ are disjoint in the sense of \cref{df:disjoint} and so this follows from \cref{disjoint_implies_conjugate}.
	\end{proof}

	\begin{lemma}\label{embeddings_are_monomorphisms}
		Scissors embeddings are monomorphisms.
	\end{lemma}

	\begin{proof}
		The proof differs from that in \cite[Proposition 2.11 (1)]{zakharevich2014} because we allowed $\varnothing$ to appear in the objects of $\cW(\cA)$, see \cref{empty_sets_in_W_discussion}. However, we are also allowed to freely pass to refinements, which changes the proof, and eliminates any potential problems because up to refinement all instances of $\varnothing$ can be eliminated.

		We assume we are given two parallel embeddings $g,h$ and a third embedding $f$:
		\[ \scr{P} \rightrightarrows \scr{Q} \rightarrow \scr{S} \]
		Taking the zig-zags for each of these maps gives covers
		\[ \scr{R}_f \to \scr{Q}, \quad \scr{R}_f \sqcup \scr{R}_{f'} \to \scr{S}, \]
		\[ \scr{R}_g \to \scr{P}, \quad \scr{R}_g \sqcup \scr{R}_{g'} \to \scr{Q}, \]
		\[ \scr{R}_h \to \scr{P}, \quad \scr{R}_h \sqcup \scr{R}_{h'} \to \scr{Q}. \]
		Passing to a common refinement of the three covers of $\scr{Q}$, we may assume that $\scr{R}_g$ and $\scr{R}_h$ are sub-tuples of $\scr{R}_f$, in other words
		\[ \scr{R}_f = \scr{R}_g \sqcup \scr{R}_{g'} = \scr{R}_h \sqcup \scr{R}_{h'}. \]

		The composite map $fg$ is then represented by
		\[ \scr{R}_g \to \scr{P}, \quad \scr{R}_g \sqcup \scr{R}_{g'} \sqcup \scr{R}_{f'} \to \scr{S}, \]
		and similarly $fh$ is represented by
		\[ \scr{R}_h \to \scr{P}, \quad \scr{R}_h \sqcup \scr{R}_{h'} \sqcup \scr{R}_{f'} \to \scr{S}, \]
		If $fg = fh$, then after further refinement, these presentations are the same, so $\scr{R}_g = \scr{R}_h$, with the same cover of $\scr{P}$. Furthermore, the maps $\scr{R}_g \sqcup \scr{R}_{g'} \to \scr{Q}$ and $\scr{R}_h \sqcup \scr{R}_{h'} \to \scr{Q}$ agree with the fixed map $\scr{R}_f \to \scr{Q}$, so they also agree. Therefore $g = h$ in $\category{UG}(\cA)$.
	\end{proof}

	\begin{corollary}
		The complement of a scissors embedding $e\colon \scr{P} \to \scr{Q}$ is well-defined up to unique isomorphism, as an object of the comma category $\category{UG}(\cA)/\scr{Q}$. The same applies to an intersection or union of finitely many scissors embeddings $e_i\colon \scr{P}_i \to \scr{Q}$.
	\end{corollary}

	It also follows that the factorisations of \cref{disjoint_iff_complement} are unique. In particular, the data of a scissors embedding $X^{\sqcup p+1} \hookrightarrow Y$ is equivalent to the data of $(p+1)$ disjoint scissors embeddings $X \hookrightarrow Y$, a fact that we will use frequently in the subsequent proofs.

	\subsubsection{Scissors embeddings and scissors congruences} 
	We end with a few results that relate scissors embeddings to scissors congruences. Of course we can compose any embedding with a congruence and get another embedding. The following results help us understand how they interact with the empty tuple $\varnothing = \{\} \cong \{\varnothing\}$.

	\begin{lemma}\label{emb_into_zero}
		$\scr{P} \simeq \varnothing$ if and only if there exists a scissors embedding $\scr{P} \to \varnothing$.
	\end{lemma}
	\begin{proof}
		This is a straightforward consequence of \cref{sc_to_empty}.
	\end{proof}

	\begin{lemma}\label{when_an_emb_is_a_cong}
		If $e\colon \scr{P} \to \scr{Q}$ is a scissors embedding and the complement $\scr{Q} - e(\scr{P})$ is scissors congruent to $\varnothing$, then $e$ defines a scissors congruence $\scr{P} \simeq \scr{Q}$.
	\end{lemma}

	\begin{proof}
		This is also a straightforward consequence of \cref{sc_to_empty}: by taking a refinement, the terms in the span making up the complement $\scr{Q} - e(\scr{P})$ can be replaced by the empty tuple, and we are left with a scissors congruence $\scr{P} \simeq \scr{Q}$.
	\end{proof}

	\begin{remark}
	This implies that the isomorphisms in $\category{UG}(\cA)$ are exactly the scissors congruences $\cG(\cA)$. In more detail, given two inverse embeddings, the complements add to $\varnothing$, and therefore by \cref{emb_into_zero} the complements are equivalent to $\varnothing$. Therefore by \cref{when_an_emb_is_a_cong} the embeddings are isomorphisms.
	\end{remark}

	\begin{lemma}\label{empty_initial_in_ug}
		The empty tuple $\varnothing$ is initial in $\category{UG}(\cA)$.
	\end{lemma}

	\begin{proof}
		By \cref{sc_emb_simplified} the only possible morphism from the empty tuple to $\scr{P}$ consists of an empty cover.
	\end{proof}
	
	Finally, recall the definition of locally standard from \cite[Definition 2.5]{randalwilliamswahl2017}. This is defined in loc.cit.~only for homogeneous categories, but it makes sense in any symmetric monoidal category in which the unit is the initial object $\varnothing$.

	\begin{lemma}\label{lem:ug-conditions}	The category $\category{UG}(\cA)$ is locally standard at any pair $(\scr{A},\scr{X})$ with $\scr{X} \not \simeq \varnothing$.
	\end{lemma}

	\begin{proof}
		We first need to know that the two obvious scissors embeddings $\scr{X} \to \scr{A} \sqcup \scr{X} \sqcup \scr{X}$ are distinct when $\scr{X} \not\simeq \varnothing$. This follows because if they were equal, their intersection would be scissors congruent to $\scr{X}$, and also to $\varnothing$ by \cref{rem:disjoint-scissors-embeddings} since they are disjoint as in \cref{example:cover-of-scissors-embeddings}; together this would imply that $\scr{X} \simeq \varnothing$.
		 We also need to know that composing with the embedding $\varnothing \to \scr{X}$ induces an injective map
		\[ \hom_{\category{UG}(\cA)}(\scr{X},\scr{A} \sqcup \scr{X}^{\sqcup (n-1)}) \longrightarrow  \hom_{\category{UG}(\cA)}(\scr{X},\scr{A} \sqcup \scr{X}^{\sqcup n}). \]
		This follows because all scissors embeddings are monomorphisms (\cref{embeddings_are_monomorphisms}).
	\end{proof}

	\section{Scissors congruences from a stability perspective}
	\label{sec:scissors-congruecne-from-a-stability-perspective}
	In this section we establish homological stability results for the groupoid $\cG(\cA)$ of an assembler $\cA$, as long as $\cA$ satisfies a small list of axioms.

	\subsection{EAZ- and S-assemblers} Though we are interested in scissors automorphism groups arising from classical geometries, the homological stability results of this section hold for any assembler satisfying one of two sets of axioms, inspired by Euclidean scissors congruences without scaling and with scaling respectively.

	\subsubsection{EAZ-assemblers} Let $(\bR,\geq)$ be a totally ordered abelian group and $\bR_{\geq 0}$ the subset of elements $r$ satisfying $r \geq 0$. In most examples $\bR$ will be given by the real numbers, and the reader is invited to keep this example in mind.

	\begin{definition}\label{volume_function}
	A \emph{volume function} on an assembler $\cA$ is a function $\vol \colon \ob(\cA) \to \bR_{\geq 0}$ such that if $\{P_i \to P\}_{i \in I}$ is a cover then $\vol(P) = \sum_{i \in I} \vol(P_i)$.
	\end{definition}

	This is similar to a measure in the sense of \cite[Definition 6.1]{bgmmz}, except that there is no group of equivariance $G$, the abelian group is ordered, and only nonnegative measures are allowed. Note that the above definition implies that $\vol(P) = \vol(Q)$ if $P \cong Q$, because isomorphisms are covers. It also implies that $\vol(\varnothing) = 0$ because $\varnothing$ admits an empty cover.

	We extend the volume function from $\cA$ to $\cG(\cA)$ by setting $\vol(\scr{P}) = \sum_{i \in I} \vol(P_i)$ for $\scr{P}=(I,\{P_i\}_{i \in I})$. It is immediate that $\scr{P} \simeq \scr{Q}$ implies $\vol(\scr{P}) = \vol(\scr{Q})$, and the existence of an embedding $\scr{P} \to \scr{Q}$ implies $\vol(\scr{P}) \leq \vol(\scr{Q})$.

	\begin{definition}\label{def:ass-vol-props} An \emph{EAZ-assembler} is an assembler $\cA$ with a volume function $\vol \colon \ob(\cA) \to \bR_{\geq 0}$ such that for any objects $\scr{P},\scr{P}',\scr{Q} \in \ob(\cG(\cA))$:
		\begin{description}
			\item[\namedlabel{enum:ass-vol-existence}{(E)}] If $\vol(\scr{P}) < \vol(\scr{Q})$, then there exists a $\scr{P}' \in \ob(\cG(\cA))$ such that $\scr{P} \sqcup \scr{P}' \simeq \scr{Q}$.
			\item[\namedlabel{enum:ass-vol-archimedean}{(A)}] 
			If $\epsilon > 0$ in $\bR$, then $\scr{Q}$ admits a cover $\{\scr{Q}_i\}$ with $2\vol(\scr{Q}_i) \leq \epsilon$ for each $i \in I$.
			\item[\namedlabel{enum:ass-vol-zylev}{(Z)}] If $\scr{P} \sqcup \scr{Q} \simeq \scr{P}' \sqcup \scr{Q}$ then $\scr{P} \simeq \scr{P}'$.
		\end{description}
	\end{definition}

	Intuitively, this means that volume determines the existence of scissors embeddings, that covers with arbitrarily small volume exist, and that scissors congruence has cancellation.

	\begin{remark} \ref{enum:ass-vol-existence} stands for ``existence of embeddings,'' \ref{enum:ass-vol-archimedean} for ``archimedean property,'' and \ref{enum:ass-vol-zylev} for ``Zylev's Theorem.'' Note that \ref{enum:ass-vol-existence} and \ref{enum:ass-vol-archimedean} are similar to axiom (VA) in \cite[Theorem 3.1]{sah1979} (Property A in \cite{zak_perspectives}), except that here the volume is quantified as a number, instead of being characterized abstractly by the existence of embeddings.
	\end{remark}

	In practice axioms \ref{enum:ass-vol-existence} and \ref{enum:ass-vol-archimedean} are easier to verify, while \ref{enum:ass-vol-zylev} is harder. To make this easier we consider the following additional axiom on an assembler with volume function:
	\begin{description}
		\item[\namedlabel{enum:ass-vol-iso-0}{(V0)}] If $\vol(\scr{P}) = 0$ then $\scr{P} \simeq \varnothing$.
	\end{description}
	
	\begin{lemma}[Zylev's Theorem]\label{eaz_reduction}
	\ref{enum:ass-vol-existence}, \ref{enum:ass-vol-archimedean}, and \ref{enum:ass-vol-iso-0} imply \ref{enum:ass-vol-zylev}.
	\end{lemma}
	
	\begin{proof}
	Assume \ref{enum:ass-vol-existence}, \ref{enum:ass-vol-archimedean}, and \ref{enum:ass-vol-iso-0}. We let $\scr{P} \sqcup \scr{Q} \simeq \scr{P}' \sqcup \scr{Q}$, and argue as in \cite[Theorem 3.1]{sah1979}. The case where $\scr{P}$ and $\scr{P}'$ are both empty is trivial, so we may assume that $\scr{P}'$ is nonempty. By \ref{enum:ass-vol-iso-0}, we thus have $\vol(\scr{P}') > 0$. Since $\scr{P} \sqcup \scr{Q} \simeq \scr{P}' \sqcup \scr{Q}$, it then follows that $\vol(\scr{P}) = \vol(\scr{P}') > 0$ and, in particular, that $\scr{P}$ is nonempty as well.
	
	To establish that $\scr{P} \simeq \scr{P}'$, we first show that we can produce a scissors congruence $\scr{P} \sqcup \scr{Q} \simeq \scr{P}' \sqcup \scr{Q}$ with the property that $\scr{P}$ and $\scr{P'}$ share a nonempty piece in common.
	To show this, let $f\colon \scr{P} \to \scr{P}' \sqcup \scr{Q}$ be the embedding arising from some scissors congruence $\bar f\colon \scr{P} \sqcup \scr{Q} \simeq \scr{P}' \sqcup \scr{Q}$. Use \ref{enum:ass-vol-archimedean} to cut up $\scr{P}$ into pieces, such that the first piece $\scr{P}_0$ has nonzero volume smaller than $\scr{P'}$. Then by \ref{enum:ass-vol-existence}, $\scr{P}_0$ admits an embedding into $\scr{P'}$. This extends to another embedding $g\colon \scr{P}_0 \to \scr{P}' \sqcup \scr{Q}$. By \cref{make_embeddings_overlap}, there is a scissors congruence $h$ of $\scr{P}' \sqcup \scr{Q}$ such that $hf|_{\scr{P}_0}$ and $g$ overlap. But since $g$ sends $\scr{P}_0$ into $\scr{P'}$, this means that $h \bar f$ is a scissors congruence sending some nontrivial piece of $\scr{P}$ into $\scr{P'}$, as desired.
	
	Now fix a scissors congruence $\scr{P} \sqcup \scr{Q} \simeq \scr{P}' \sqcup \scr{Q}$ along which the intersection $\scr{P} \cap \scr{P'}$ is nonempty. By \ref{enum:ass-vol-iso-0}, it therefore has positive volume. By \ref{enum:ass-vol-archimedean}, we can take a cover of $\scr{Q}$ in which the pieces $\scr{Q}_i$ have volume smaller than the volume of the intersection $\scr{P} \cap \scr{P'}$. By \ref{enum:ass-vol-existence}, therefore each piece $\scr{Q}_i$ embeds into the intersection.
	
	We show inductively that we can remove each piece $\scr{Q}_i$ from $\scr{Q}$ while preserving the fact that $\scr{P} \sqcup \scr{Q} \simeq \scr{P}' \sqcup \scr{Q}$. After finitely many steps, this shows that $\scr{P} \simeq \scr{P}'$. To remove one piece $\scr{Q}_0$, let $\scr{Q'}$ be the union of the remaining pieces. We then perform the following scissors congruences, which are also illustrated in \cref{fig:zylev}.
	\begin{figure}[p]
	\vspace{1em}
	\centerline{
	\def\svgwidth{4.0in}
	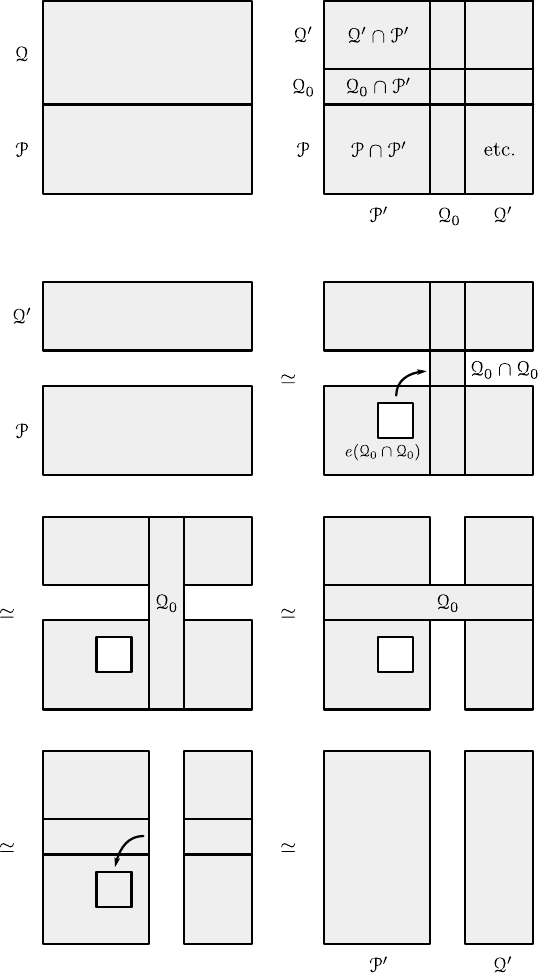
	}
	\caption{Proof of Zylev's Theorem.\label{fig:zylev}}
	\end{figure}

		\begin{align*}
			\scr{P} \sqcup \scr{Q'} &\simeq (\scr{P} \cap \scr{P'}) \sqcup (\scr{P} \cap \scr{Q}_0) \sqcup (\scr{P} \cap \scr{Q'}) \sqcup (\scr{Q'} \cap \scr{P'}) \sqcup (\scr{Q'} \cap \scr{Q}_0) \sqcup (\scr{Q'} \cap \scr{Q'}) \\
			&\simeq (\scr{P} \cap \scr{P'} - e(\scr{Q}_0 \cap \scr{Q}_0)) \sqcup (\scr{Q}_0 \cap \scr{Q}_0) \\
			&\qquad \sqcup (\scr{P} \cap \scr{Q}_0) \sqcup (\scr{P} \cap \scr{Q'}) \sqcup (\scr{Q'} \cap \scr{P'}) \sqcup (\scr{Q'} \cap \scr{Q}_0) \sqcup (\scr{Q'} \cap \scr{Q'}) \\
			&\simeq (\scr{P} \cap \scr{P'} - e(\scr{Q}_0 \cap \scr{Q}_0)) \sqcup \scr{Q}_0 \sqcup (\scr{P} \cap \scr{Q'}) \sqcup (\scr{Q'} \cap \scr{P'}) \sqcup (\scr{Q'} \cap \scr{Q'}) \\
			&\simeq (\scr{P} \cap \scr{P'} - e(\scr{Q}_0 \cap \scr{Q}_0)) \sqcup (\scr{Q}_0 \cap \scr{Q}_0) \\
			&\qquad \sqcup (\scr{Q}_0 \cap \scr{P'}) \sqcup (\scr{Q'} \cap \scr{P'}) \sqcup (\scr{P} \cap \scr{Q'}) \sqcup (\scr{Q}_0 \cap \scr{Q'}) \sqcup (\scr{Q'} \cap \scr{Q'}) \\
			&\simeq (\scr{P} \cap \scr{P'}) \sqcup (\scr{Q}_0 \cap \scr{P'}) \sqcup (\scr{Q'} \cap \scr{P'}) \sqcup (\scr{P} \cap \scr{Q'}) \sqcup (\scr{Q}_0 \cap \scr{Q'}) \sqcup (\scr{Q'} \cap \scr{Q'}) \\
			&\simeq \scr{P'} \sqcup \scr{Q'}.
		\end{align*}
		In the first step we expand each term as the union of three intersections. Then we remove the embedded copy of $\scr{Q}_0 \cap \scr{Q}_0$ from $\scr{P} \cap \scr{P'}$. Then we group three terms together to form $\scr{Q}_0$. Then we expand out $\scr{Q}_0$ again as the union of three terms in a different way. Finally, we put the embedded copy of $\scr{Q}_0 \cap \scr{Q}_0$ back into $\scr{P} \cap \scr{P'}$, and collect the terms together to get $\scr{P'} \sqcup \scr{Q'}$.
	\end{proof}
	
	\begin{remark}
		Axiom \ref{enum:ass-vol-iso-0} is slightly stronger than an axiom called (V) in \cite{sah1979} and \cite{zak_perspectives}, which says that if $\scr{P} \simeq \scr{Q}$, and $e\colon \scr{P} \to \scr{Q}$ is a scissors embedding, then $e$ is a scissors congruence. This axiom (V) is not as compatible with the arguments in this paper since we are relying on an explicit volume function, rather than the abstract notion of ``larger than twice'' that appears in \cite{sah1979} and \cite{zak_perspectives}.
	\end{remark}

	\begin{example}
		In \cref{sec:ea} we will define the assembler of polytopes in Euclidean, spherical, or hyperbolic geometry, and prove that it satisfies axioms \ref{enum:ass-vol-existence}, \ref{enum:ass-vol-archimedean}, and \ref{enum:ass-vol-iso-0}, and therefore also \ref{enum:ass-vol-zylev}.
	\end{example}

	\subsubsection{S-assemblers} We also consider a second set of axioms, satisfied by the assembler of polytopes in Euclidean space where scaling is allowed.

	\begin{definition} \label{def:s-assembler} An \emph{S-assembler} is an assembler $\cA$  such that:
	\begin{description}
		\item[\namedlabel{enum:ass-s}{(S)}] For all $\scr{P}$ and all $\scr{Q} \not\simeq \varnothing$, there exists a $\scr{P}' \not\simeq \varnothing$ such that $\scr{P} \sqcup \scr{P}' \simeq \scr{Q}$.
	\end{description}
	\end{definition}

	Intuitively, this means that any object can be made to strictly fit inside any other nonempty object.

	\begin{remark} \ref{enum:ass-s} stands for ``squeezing''.
	\end{remark}

	\begin{example}In \cref{sec:ea} we will see that the assembler of polytopes in Euclidean geometry satisfies axiom \ref{enum:ass-s} if we additionally allow scaling.
	\end{example}
	
	\begin{lemma}\label{s-implies-wz}
		The axiom \ref{enum:ass-s} implies the following weak form of Zylev's Theorem:
			\begin{description}
			\item[\namedlabel{enum:ass-weak-zylev}{(WZ)}]  If $\scr{P} \sqcup \scr{Q} \simeq \scr{P}' \sqcup \scr{Q}$ and $\scr{P},\scr{P'} \not\simeq \varnothing$, then $\scr{P} \simeq \scr{P}'$.
		\end{description}
	\end{lemma}

	\begin{proof}
	 	We again argue as in \cite[Theorem 3.1]{sah1979}, but now the proof is simpler than in \cref{eaz_reduction}. As in that proof, we first show that we can produce a scissors congruence $\scr{P} \sqcup \scr{Q} \simeq \scr{P}' \sqcup \scr{Q}$ with the property that $\scr{P}$ and $\scr{P'}$ share a nonempty piece in common. 
	 
	 	To show this, let $f\colon \scr{P} \to \scr{P}' \sqcup \scr{Q}$ be the embedding arising from some scissors congruence $\bar f\colon \scr{P} \sqcup \scr{Q} \simeq \scr{P}' \sqcup \scr{Q}$. Use \ref{enum:ass-s} to produce another scissors congruence $\scr{P} \sqcup \scr{P''} \simeq \scr{P'}$, which gives another embedding $g\colon \scr{P} \to \scr{P}' \sqcup \scr{Q}$. By \cref{make_embeddings_overlap}, there is a scissors congruence $h$ of $\scr{P}' \sqcup \scr{Q}$ such that $hf$ and $g$ overlap. But since $g$ sends $\scr{P}$ into $\scr{P'}$, this means that $h \bar f$ is a scissors congruence sending some nontrivial piece of $\scr{P}$ into $\scr{P'}$, as desired.
	 	
	 	Now that we have such a scissors congruence $\scr{P} \sqcup \scr{Q} \simeq \scr{P}' \sqcup \scr{Q}$, we divide $\scr{P} \sqcup \scr{Q}$ into pieces corresponding to the four intersections $\scr{P} \cap \scr{P'}$, $\scr{P} \cap \scr{Q}$, $\scr{Q} \cap \scr{P'}$, $\scr{Q} \cap \scr{Q}$. We pick an embedding of $\scr{Q} \cap \scr{Q}$ into $\scr{P} \cap \scr{P'}$, which is possible by \ref{enum:ass-s} because $\scr{P} \cap \scr{P'}$ is nonempty. Then
	 	\begin{align*}
			\scr{P} &\simeq (\scr{P} \cap \scr{P'}) \sqcup (\scr{P} \cap \scr{Q}) \\
			 &\simeq (\scr{P} \cap \scr{P'} - e(\scr{Q} \cap \scr{Q})) \sqcup (\scr{Q} \cap \scr{Q}) \sqcup (\scr{P} \cap \scr{Q}) \\
			 &\simeq (\scr{P} \cap \scr{P'} - e(\scr{Q} \cap \scr{Q})) \sqcup \scr{Q} \\
			 &\simeq (\scr{P} \cap \scr{P'} - e(\scr{Q} \cap \scr{Q})) \sqcup (\scr{Q} \cap \scr{Q}) \sqcup (\scr{Q} \cap \scr{P'}) \\
			 &\simeq (\scr{P} \cap \scr{P'}) \sqcup (\scr{Q} \cap \scr{P'}) \\
			&\simeq \scr{P'}.\qedhere
		\end{align*}
	\end{proof}

	\subsection{Connectivity of the complex of disjoint embeddings}
	Let $\cA$ be any assembler. We may define various embedding complexes for the symmetric monoidal category $\category{UG}(\cA)$ as in \cite[Section 2]{randalwilliamswahl2017}. To study these, we define the following:

	\begin{definition}\label{embedding_complex_1}
		Let $\scr{B}$ and $\scr{X}$ be objects in $\category{UG}(\cA)$. The simplicial complex $E(\scr{B},\scr{X})$ has vertices given by scissors embeddings $v\colon \scr{X} \hookrightarrow \scr{B}$ whose complement (\cref{def:complement}) is nonempty, meaning that $\scr{B} - v(\scr{X}) \not\simeq \varnothing$ (\cref{empty_meaning}). A $(p+1)$-tuple of vertices $(v_0,\ldots,v_p)$ forms a $p$-simplex if they are disjoint and if the complement of their union (\cref{def:intersection-and-union}) is nonempty, so $\scr{B} - (\cup_i v_i(\scr{X})) \not\simeq \varnothing$.
	\end{definition}

	Note that by \cref{disjoint_iff_complement}, the data of a $(p+1)$-tuple of disjoint scissors embeddings $\scr{X} \hookrightarrow \scr{B}$ is equivalent to the data of a scissors embedding $\scr{X}^{\sqcup p+1} \hookrightarrow \scr{B}$. 
	
	The homological stability machinery will be powered by the connectivity of the simplicial complex $E(\scr{B},\scr{X})$. To establish this connectivity we use property \ref{enum:ass-vol-existence} or \ref{enum:ass-s} from the previous section.

	\begin{theorem}
		\label{thm:connectivity-of-complexes-of-destabilisation}\,
		\begin{enumerate}
			\item \label{enum:connectivity-i} Suppose $\cA$ satisfies \ref{enum:ass-vol-existence}. Let $\scr{X},\scr{B} \in \category{UG}(\cA)$ and suppose there is a $k \in \mathbb{N}$ such that
			\[ 0 \leq k\vol(\scr{X}) < \vol(\scr{B}). \]
			Then $E(\scr{B},\scr{X})$ is $\lfloor \frac{k-3}{2} \rfloor$-connected.
			\item \label{enum:connectivity-ii}  Suppose $\cA$ satisfies \ref{enum:ass-s}. Let $\scr{X},\scr{B} \in \category{UG}(\cA)$ such that $\scr{B} \not\simeq \varnothing$. Then $E(\scr{B},\scr{X})$ is contractible.
		\end{enumerate}
	\end{theorem}

	\begin{proof}We first prove part \ref{enum:connectivity-i}. If we let $k(z) \coloneqq 2z+3$, the claim is that for any $z \geq -1$, if $\vol(\scr{B}) > k(z)\vol(\scr{X})$, then $E(\scr{B},\scr{X})$ is $z$-connected.

	We prove this by induction on the integer $z$, simultaneously for all $\scr{B}$. The base case is for $z = -1$, so $\vol(\scr{B}) > k(-1)\vol(\scr{X})=\vol(\scr{X})$. We claim that $E(\scr{B},\scr{X})$ is $(-1)$-connected, i.e.~nonempty; this holds since property \ref{enum:ass-vol-existence} implies the existence of a vertex in $E(\scr{B},\scr{X})$. The complement of the embedding $\scr{X} \to \scr{B}$ must be nonempty since it has positive volume.

	Now let $z > -1$ be an integer. For the induction step we assume that $\vol(\scr{B}) > k(z)\vol(\scr{X})$ and that the claim holds for $z' < z$. Let $\phi \colon S^d \to E(\scr{B},\scr{X})$ be a simplicial map where $d \leq z$ and $S^d$ is given the structure of a combinatorial manifold. Using property \ref{enum:ass-vol-existence}, fix a vertex $v_0 \colon \scr{X} \hookrightarrow \scr{B}$ of $E(\scr{B},\scr{X})$, and call a simplex $\Delta$ of $S^d$ \emph{bad} if every vertex $x \in \Delta$ has the property that $\phi(x) \colon \scr{X} \hookrightarrow \scr{B}$ is not disjoint from $v_0 \colon \scr{X} \hookrightarrow \scr{B}$.

	Assume that $\Delta$ is a bad simplex of maximal dimension $p$ in $S^d$, and let $\phi(\Delta) = \{w_0, \dots, w_p\}$ be the image of this bad simplex in $E(\scr{B},\scr{X})$. Since $S^d$ is a combinatorial manifold, $\Link{S^d}{\Delta}$ is a combinatorial sphere of dimension $(d - p - 1)$. For any vertex $x$ in $\Link{S^d}{\Delta}$, the simplex $x * \Delta$ is not bad, because $\Delta$ has maximal dimension, so $x$ must go to a vertex $w = \phi(x)$ such that $w$ is disjoint from $v_0$. By assumption $\{w, w_0, \dots, w_p\}$ is a simplex in $E(\scr{B},\scr{X})$, so $w$ is also disjoint from $w_i$ for $0 \leq i \leq p$, in addition to being disjoint from $v_0$.
	
	\vspace{1em}
	\centerline{
	\def\svgwidth{2.9in}
\begingroup%
  \makeatletter%
  \providecommand\color[2][]{%
    \errmessage{(Inkscape) Color is used for the text in Inkscape, but the package 'color.sty' is not loaded}%
    \renewcommand\color[2][]{}%
  }%
  \providecommand\transparent[1]{%
    \errmessage{(Inkscape) Transparency is used (non-zero) for the text in Inkscape, but the package 'transparent.sty' is not loaded}%
    \renewcommand\transparent[1]{}%
  }%
  \providecommand\rotatebox[2]{#2}%
  \newcommand*\fsize{\dimexpr\f@size pt\relax}%
  \newcommand*\lineheight[1]{\fontsize{\fsize}{#1\fsize}\selectfont}%
  \ifx\svgwidth\undefined%
    \setlength{\unitlength}{178.16855211bp}%
    \ifx\svgscale\undefined%
      \relax%
    \else%
      \setlength{\unitlength}{\unitlength * \real{\svgscale}}%
    \fi%
  \else%
    \setlength{\unitlength}{\svgwidth}%
  \fi%
  \global\let\svgwidth\undefined%
  \global\let\svgscale\undefined%
  \makeatother%
  \begin{picture}(1,0.5398249)%
    \lineheight{1}%
    \setlength\tabcolsep{0pt}%
    \put(0,0){\includegraphics[width=\unitlength,page=1]{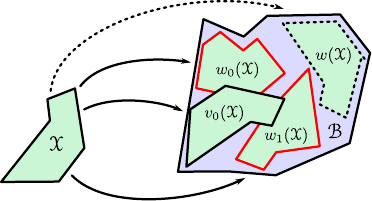}}%
  \end{picture}%
\endgroup%

	}
	\vspace{1em}
	
	To establish that the complement of $w$ in $\scr{B} - (\cup_i w_i(\scr{X})) - v_0(\scr{X})$ is nonempty, it suffices to show that it has positive volume. Note that in an assembler with a volume function as in \cref{volume_function}, volume is sub-additive along unions. Therefore
	\begin{align*}
	\vol\big(\scr{B} - (\cup_i w_i(\scr{X})) - v_0(\scr{X})\big)
	&= \vol(\scr{B}) - \vol((\cup_i w_i(\scr{X})) \cup v_0(\scr{X})) \\
	&\geq \vol(\scr{B}) - \sum_i \vol w_i(\scr{X})- \vol v_0(\scr{X}) \\
	&= \vol(\scr{B}) - (p+2)\vol(\scr{X}) \\
	&> k(z)\vol(\scr{X}) - (p+2)\vol(\scr{X}) \\
	&= (2z - p + 1)\vol(\scr{X}).
	\end{align*}
	Since $0 \leq p \leq d \leq z$, this last quantity is at least $\vol(\scr{X}) = \vol(w(\scr{X}))$, so the complement of $w(\scr{X})$ in $\scr{B} - (\cup_i w_i(\scr{X})) - v_0(\scr{X})$ must be nonempty. In summary, $w$ is a vertex of the full simplicial subcomplex
	\[E(\scr{B} - (\cup_i w_i(\scr{X})) - v_0(\scr{X}),\scr{X}) \subseteq E(\scr{B},\scr{X}). \]
	More generally, $\phi$ induces a map $\Link{S^d}{\Delta} \to E(\scr{B} - (\cup_i w_i(\scr{X})) - v_0(\scr{X}),\scr{X})$, because each simplex in this link goes to at most $(d-p-1)$ disjoint embeddings into $\scr{B} - (\cup_i w_i(\scr{X})) - v_0(\scr{X})$, whose total volume is no more than
	\[ (d-p-1)\vol(\scr{X}) \leq (2z - p + 1)\vol(\scr{X}) < \vol\big(\scr{B} - (\cup_i w_i(\scr{X})) - v_0(\scr{X})\big), \]
	and so their complement must be nonempty.
	
	Next we observe that the volume of $\scr{B} - (\cup_i w_i(\scr{X})) - v_0(\scr{X})$ is strictly greater than
	\[ (2z - p + 1)\vol(\scr{X}) \geq (2z - 2p + 1)\vol(\scr{X}) = k(z-p-1)\vol(\scr{X}). \]
	Therefore, by the induction hypothesis, the complex $E(\scr{B}- (\cup_i w_i(\scr{X})) - v_0(\scr{X}),\scr{X})$ is $(z-p-1)$-connected. This is greater than or equal to $(d-p-1)$, the dimension of $\Link{S^d}{\Delta}$, and so $\phi$ when restricted to $\Link{S^d}{\Delta}$ is null homotopic via some map
	\[ \psi \colon D(\Link{S^d}{\Delta}) \longrightarrow E(\scr{B} - (\cup_i w_i(\scr{X})) - v_0(\scr{X}),\scr{X}), \]
	for $D(\Link{S^d}{\Delta})$ a combinatorial $(d-p)$-disc. We define a simplicial map from the join $\Delta \ast D(\Link{S^d}{\Delta})$ to $E(\scr{B},\scr{X})$ given by $\phi$ on the vertices of $\Delta$ and $\psi$ on the vertices of $D(\Link{S^d}{\Delta})$; this is well-defined because the embedding $w$ corresponding to each vertex of $D(\Link{S^d}{\Delta})$ is disjoint from the embeddings $w_i$ for $0 \leq i \leq p$, and every simplex of such embeddings $w$ has nonempty complement in $\scr{B} - (\cup_i w_i(\scr{X})) - v_0(\scr{X})$, which implies that even after adding in the embeddings $w_i$, the complement in $\scr{B}$ is still nonempty.

	Note that $\Delta \ast D(\Link{S^d}{\Delta})$ is a $(d+1)$-dimensional disc, whose boundary sphere is composed of two $d$-dimensional discs $\Delta \ast \Link{S^d}{\Delta}$ and $\partial \Delta \ast D(\Link{S^d}{\Delta})$, meeting at the $(d-1)$-sphere $\partial \Delta \ast \Link{S^d}{\Delta}$. Therefore the maps from these two $d$-discs into $E(\scr{B},\scr{X})$ are homotopic relative to the boundary. One of these $d$-discs $\Delta \ast \Link{S^d}{\Delta}$ is the star of the bad simplex $\Delta$ inside our original combinatorial manifold $S^d$. So we replace it with the other disc $\partial \Delta \ast D(\Link{S^d}{\Delta})$ and get a new combinatorial manifold $d$-sphere, whose map into $E(\scr{B},\scr{X})$ is homotopic to the original map. Furthermore, in this new disc that we have added, the bad simplices all have dimension smaller than that of $\Delta$, since they do not contain all of the vertices of $\Delta$, and the remaining vertices in $D(\Link{S^d}{\Delta})$ are all good.

	Iterating this procedure removes all of the bad simplices, giving a homotopy of $\phi$ to a map that lands in the star of $v_0$. One final homotopy in this star deforms the map $\phi$ to a constant map, proving the desired connectivity.
	
	\medskip
	
	The proof for part \ref{enum:connectivity-ii} is similar, proceeding by induction over the connectivity. Supposing we have established that $E(\scr{B},\scr{X})$ is $(d-1)$-connected for all $\scr{B} \not \simeq \varnothing$, we let $\phi \colon S^d \to E(\scr{B},\scr{X})$ be a simplicial map where $S^d$ is a combinatorial manifold and will explain why it is homotopic to a map landing in the star of a simplex. This implies that $\phi$ is null-homotopic.
	
	Using property \ref{enum:ass-s}, we fix a vertex $v_0 \colon \scr{X} \hookrightarrow \scr{B}$ of $E(\scr{B},\scr{X})$. We claim that we may choose $v_0$ so that for any simplex $\Delta$ of $S^d$ the union (\cref{def:intersection-and-union}) of $v_0 \colon \scr{X} \hookrightarrow \scr{B}$ and the embeddings $w_i \colon \scr{X} \hookrightarrow \scr{B}$ corresponding to the vertices of $\Delta$ has non-empty complement in $\scr{B}$. If $\scr{X} \simeq \varnothing$, this automatically holds for all simplices $\Delta$ (using that $\scr{B} \not\simeq \varnothing$ by assumption). If $\scr{X} \not\simeq \varnothing$, we call $\Delta$ an \emph{inconvenient} simplex if it does \emph{not} have this property. However, it is not difficult to inductively rechoose $v_0$, so that there are no inconvenient simplices. A simplex $\Delta$ with $\phi(\Delta) = \{w_0,\ldots,w_p\}$ by definition satisfies that $\cup_{i} w_i(\scr{X})$ has non-empty complement in $\scr{B}$. If $\Delta$ is inconvenient, there exists a scissors embedding of this complement $\scr{B}-\cup_{i} w_i(\cX) \hookrightarrow \scr{B}$ factoring through the fixed vertex $v_0 \colon \scr{X} \hookrightarrow \scr{B}$. But using property \ref{enum:ass-s} we can pick a scissors embedding $\scr{X} \hookrightarrow \scr{B}-\cup_{i} w_i(\cX)$ with non-empty complement. Replacing $v_0$ with the scissors embedding $\scr{X} \hookrightarrow \scr{B}-\cup_{i} w_i(\cX) \hookrightarrow \scr{B}$, we have arranged that $\Delta$ is no longer inconvenient without creating any new inconvenient simplices.
	
	Having made a convenient choice of vertex $v_0 \colon \scr{X} \hookrightarrow \scr{B}$ of $E(\scr{B},\scr{X})$, we proceed as in the previous case, except that that we replace the condition that $\vol > 0$ by $\not\simeq \varnothing$ and remove all mentions of volume. The convenient choice guarantees that for a bad simplex $\Delta$ with $\phi(\Delta) = \{w_0,\ldots,w_p\}$ it is the case that $\scr{B}-\cup_i w_i(\scr{X})-v_0(\scr{X})$ is non-empty. Hence $E(\scr{B}-\cup_i w_i(\scr{X})-v_0(\scr{X}),\scr{X})$ is $(d-1)$-connected and as above we can deform $\phi$ to another simplicial map with one less bad simplex. Each simplex in the homotopy is a join of a subset of the $w_i$ with a simplex in $E(\scr{B}-\cup_i w_i(\scr{X})-v_0(\scr{X}),\scr{X})$, which gives additional embeddings that by assumption leave a nonempty complement after taking their union with $v_0(\scr{X})$ and $\cup_i w_i(\scr{X})$. Therefore each simplex in the homotopy has the property that the union of its embeddings has nonempty complement and the homotopy does not create any new inconvenient simplices.
	\end{proof}
	
	\subsection{Stability for EAZ-assemblers}
	
	We now employ the homological stability machinery from \cite{randalwilliamswahl2017}, applying it to the category of scissors embeddings $\cU\cG(\cA)$ for an EAZ-assembler $\cA$. We start by verifying that $\cU\cG(\cA)$ is homogeneous in the sense of \cite[Definition 1.3]{randalwilliamswahl2017}:
	
	\begin{lemma}\label{lem:ug-conditions-hom-eaz} If $\cA$ is an assembler with property \ref{enum:ass-vol-zylev}, then $\category{UG}(\cA)$ is homogeneous.
	\end{lemma}
	
	\begin{proof}
		To establish this lemma, we note that $\category{G}(\cA)$ has no zero divisors by \cref{emb_into_zero} (i.e.\ if $\scr{X} \sqcup \scr{Y} \simeq \varnothing$, then $\scr{X}, \scr{Y} \simeq \varnothing$), and verify item (c) and (d) in \cite[Theorem 1.10]{randalwilliamswahl2017}. Item (c) is precisely property \ref{enum:ass-vol-zylev}, which is true by assumption, and item (d) asserts that the map
		\[ (-)\sqcup \id_\scr{Q} \colon \aut_{\cG(\cA)}(\scr{P}) \longrightarrow  \aut_{\cG(\cA)}(\scr{P} \sqcup \scr{Q}) \]
		is injective. This is true, because by precomposing $f \sqcup \id_\scr{Q}$ with the scissors embedding $\iota_\scr{P}$ we obtain $\iota_\scr{P} \circ f$. Since $\iota_\scr{P}$ is a monomorphism (\cref{embeddings_are_monomorphisms}), the map $\iota_\scr{P} \circ f$ determines $f$.
	\end{proof}
	
	\cref{lem:ug-conditions-hom-eaz} verifies the conditions \textbf{H1} and \textbf{H2} from \cite[Definition 1.3]{randalwilliamswahl2017}, so to apply the results in \cite[Section 3]{randalwilliamswahl2017} it remains to verify the condition \textbf{LH3} \cite[Definition 2.2]{randalwilliamswahl2017} (the conditions \textbf{LH1} and \textbf{LH2} follow from \textbf{H1} and \textbf{H2}, respectively). This concerns, for $\scr{A}$ and $\scr{X}$ objects in $\cG(\cA)$, a collection of semi-simplicial sets $W_n(\scr{A},\scr{X})$ constructed from the category $\cU\cG(\cA)$ \cite[Definition 2.1]{randalwilliamswahl2017}:
	
	\begin{definition}\label{def:wn} Let $W_n(\scr{A},\scr{X})$ be the semi-simplicial set with $p$-simplices
		\[ \hom_{\cU\cG(\cA)}(\scr{X}^{\sqcup p+1}, \scr{A} \sqcup \scr{X}^{\sqcup n}), \]
		and face maps given by composing with the inclusions of the form $\scr{X}^{\sqcup p} \to \scr{X}^{\sqcup p+1}$.\end{definition}
		
	Condition \textbf{LH3} then asks that $W_n(\scr{A},\scr{X})$ is highly connected, but by the subsequent \cref{s_to_w} we may equivalently consider the following simplicial complex $S_n(\scr{A},\scr{X})$ \cite[Definition 2.8]{randalwilliamswahl2017}:
		
	\begin{definition}\label{def:sn} Let $S_n(\scr{A},\scr{X})$ be the simplicial complex with vertices $W_n(\scr{A},\scr{X})_0$, where a $(p+1)$-tuple forms a $p$-simplex if and only if there exists a simplex in $W_n(\scr{A},\scr{X})_p$ with these vertices.
	\end{definition}

	\begin{lemma}\label{s_to_w}
		Suppose $\cA$ is an assembler with property \ref{enum:ass-vol-zylev}, $\scr{X}$ is an object of $\category{UG}(\cA)$, and $a, k \geq 1$. Then the simplicial complex $S_n(\scr{A},\scr{X})$ is $\lfloor \frac{n-a}{k} \rfloor$-connected for all $n \geq 0$ if and only if the semi-simplicial set $W_n(\scr{A},\scr{X})$ is $\lfloor \frac{n-a}{k} \rfloor$-connected for all $n \geq 0$.
	\end{lemma}
	
	\begin{proof}
		If $\scr{X} \simeq \varnothing$, then \cref{empty_initial_in_ug} implies that $S_n(\scr{A},\scr{X})$ is the one-point space and $W_n(\scr{A},\scr{X})$ has exactly one $p$-simplex for every $p \geq 0$. So both complexes are contractible, and the claim is trivially true. Otherwise $\scr{X} \not\simeq \varnothing$, so the category $\cU\cG(\cA)$ is homogeneous and locally standard at $(\scr{A},\scr{X})$ by \cref{lem:ug-conditions} and \cref{lem:ug-conditions-hom-eaz}. Hence, \cite[Proposition 2.9]{randalwilliamswahl2017} and \cite[Theorem 2.10]{randalwilliamswahl2017} imply the claim.
	\end{proof}
	
	Unwinding the definitions, a $(p+1)$-tuple of embeddings $\scr{X} \to \scr{A} \sqcup \scr{X}^{\sqcup n}$ forms a $p$-simplex in $S_n(\scr{A},\scr{X})$ if and only if they form a scissors embedding $\scr{X}^{\sqcup p+1} \to \scr{A} \sqcup \scr{X}^{\sqcup n}$, which is the case if and only if they are disjoint. This is almost equal to one of the simplicial complexes from \cref{embedding_complex_1}: $E(\scr{A} \sqcup \scr{X}^{\sqcup n},\scr{X})$ is equal to the subcomplex of $S_n(\scr{A},\scr{X})$ where we require the complement of a $p$-simplex to be non-empty. 
	
	\begin{lemma}\label{lem:aez-wn-connectivity} 
		Suppose $\cA$ is an assembler with volume function satisfying properties \ref{enum:ass-vol-existence} and \ref{enum:ass-vol-zylev}. If $\vol(\scr{A})> \vol(\scr{X}) \geq 0$, then $W_n(\scr{A},\scr{X})$ and $S_n(\scr{A},\scr{X})$ are at least $\lfloor \frac{n-2}{2} \rfloor$-connected for all $n \geq 0$.
	\end{lemma}
	
	\begin{proof}Because the volume of the complement of a $p$-simplex $\scr{X}^{\sqcup p+1} \hookrightarrow \scr{A} \sqcup \scr{X}^{\sqcup n}$ for $p \leq n$ is at least $\vol(\scr{A}) - \vol(\scr{X}) >0$, the inclusion $E(\scr{A} \sqcup \scr{X}^{\sqcup n},\scr{X}) \subseteq S_n(\scr{A},\scr{X})$ is an equality on $n$-skeleta. Thus $S_n(\scr{A},\scr{X})$ is at least $\lfloor \frac{n-2}{2} \rfloor$-connected for all $n \geq 0$ by \cref{thm:connectivity-of-complexes-of-destabilisation} \ref{enum:connectivity-i} and the estimate $(n+1)\vol(\scr{X}) < \vol(\scr{A})+n\vol(\scr{X}) = \vol(\scr{A}\sqcup \scr{X}^{\sqcup n})$. By \cref{s_to_w} for $k=2$ and $a=2$, $W_n(\scr{A},\scr{X})$ is also at least $\lfloor \frac{n-2}{2} \rfloor$-connected for all $n \geq 0$.\end{proof}
	
	With our preparation complete, we now apply the machinery:
	
	\begin{theorem}\label{thm:embeddings-induce-homology-isomorphisms-vol}
		Let $\cA$ be an EAZ-assembler. Then for any scissors embedding $\scr{P} \hookrightarrow \scr{Q}$ with $\vol(\scr{P})>0$ the map
		\[\aut_{\category{UG}(\cA)}(\scr{P}) \longrightarrow \aut_{\category{UG}(\cA)}(\scr{Q})\] 
		induces an isomorphism on homology, and the same is true with abelian local coefficients. These isomorphisms are independent of the choice of scissors embedding.
	\end{theorem}
	
	\begin{proof}
		We first prove that if $\cA$ is an EAZ-assembler and $\scr{A},\scr{X} \in \cG(\cA)$ with $\vol(\scr{A})>0$,
		then the map \[\aut_{\category{UG}(\cA)}(\scr{A}) \longrightarrow \aut_{\category{UG}(\cA)}(\scr{A} \sqcup \scr{X})\]
		induces an isomorphism on homology with constant or abelian local coefficients. The first part then follows after noting that a scissors embedding $e \colon \scr{P} \hookrightarrow \scr{Q}$ by definition allows us to write $\scr{Q} \simeq \scr{P} \sqcup \scr{P}'$.
		
		We start with a summary of the argument. Note that by \cref{when_an_emb_is_a_cong} the map $\scr{A} \hookrightarrow \scr{A} \sqcup \scr{X}$ is an isomorphism if $\scr{X} \simeq \varnothing$. Hence, we may assume without loss of generality that $\scr{X} \not\simeq \varnothing$. The idea is that when $\vol(\scr{X}) < \vol(\scr{A})$, the connectivity results proven above can be combined with \cite[Theorem 3.1]{randalwilliamswahl2017} show that the map $\aut(\scr{A} \sqcup \scr{X}^{\sqcup n}) \to \aut(\scr{A} \sqcup \scr{X}^{\sqcup n + 1})$ is an isomorphism on homology in a range of degrees tending to infinity with $n$. If we cut $\scr{X}$ into smaller pieces, however, the slope of that range increases. If we then remove those pieces from $\scr{A}$ and start with the resulting smaller object $\scr{A'}$, the offset of that range increases. In this way we can arrange that $\aut(\scr{A}) \to \aut(\scr{A} \sqcup \scr{X})$ is an isomorphism in homology in any degree that we want. 
		
		\medskip
		
		To be more specific, fix a degree $q \geq 0$. By iterating property \ref{enum:ass-vol-archimedean}, for any $k \geq 1$ we can find a scissors congruence $\bigsqcup_{i=1}^m \scr{X}_i \congto \scr{X}$ with $2^k\vol(\scr{X}_i) \leq \vol(\scr{A})$. By taking $k$ sufficiently large, we can assume that $(2q+2) \vol(\scr{X}_i)<\vol(\scr{A})$. It suffices to show that adjoining each $\scr{X}_i$ in turn induces an isomorphism on $q$th homology, so without loss of generality we now assume that
		\[ 0 \leq (2q+2) \vol(\scr{X})<\vol(\scr{A}). \]
		By \ref{enum:ass-vol-existence}, there exists a scissors embedding $\scr{X}^{ \sqcup (2q+2)} \to \scr{A}$. Let $\scr{A'}$ denote the complement of the first $(2q+1)$ copies of $\scr{X}$. Then $\vol(\scr{A'}) > \vol(\scr{X}) \geq 0$. By \cref{lem:aez-wn-connectivity} the semi-simplicial set $W_n(\scr{A'},\scr{X})$ is $\lfloor \frac{n-2}{2} \rfloor$-connected and thus the pair $(\scr{A'},\scr{X})$ satisfies condition \textbf{LH3} in \cite[Definition 2.2]{randalwilliamswahl2017}. Using that $(\scr{A'},\scr{X})$ is also locally homogeneous by \cref{lem:ug-conditions-hom-eaz}, by \cite[Theorem 3.1]{randalwilliamswahl2017} the map
		\[\aut(\scr{A'} \sqcup \scr{X}^{\sqcup n}) \longrightarrow \aut(\scr{A'} \sqcup \scr{X}^{\sqcup n + 1})\]
		induces an isomorphism on $q$th homology when $q \leq \frac{n-1}{2}$. Taking $n = 2q+1$, we get an isomorphism on $q$th homology for the map
		\[\aut(\scr{A'} \sqcup \scr{X}^{\sqcup 2q+1}) \longrightarrow \aut(\scr{A'} \sqcup \scr{X}^{\sqcup 2q+2}). \]
		However, $\scr{A} \simeq \scr{A'} \sqcup \scr{X}^{\sqcup 2q+1}$ in the category $\cG(\cA)$, so this is isomorphic to the desired map $\aut(\scr{A}) \to \aut(\scr{A} \sqcup \scr{X})$.
		
		To prove the analogous statement with abelian local coefficients, we instead arrange for $\scr{A}$ to contain $(3q+4)$ copies of $\scr{X}$, and let $\scr{A'}$ denote the complement of a scissors embedding $\scr{X}^{ \sqcup (3q+3)} \to \scr{A}$. Then $W_n(\scr{A'},\scr{X})$ is $\lfloor \frac{n-2}{2} \rfloor$-connected as before, and therefore also $\lfloor \frac{n-2}{3} \rfloor$-connected, for all $n \geq 0$. By \cite[Theorem 3.4]{randalwilliamswahl2017}, the map $\aut(\scr{A'} \sqcup \scr{X}^{\sqcup n}) \to \aut(\scr{A'} \sqcup \scr{X}^{\sqcup n + 1})$ is therefore an isomorphism on $q$th homology with abelian local coefficients when $q \leq \frac{n-3}{3}$. Taking $n = 3q+3$, we conclude that $\aut(\scr{A}) \to \aut(\scr{A} \sqcup \scr{X})$ is an isomorphism on $q$th homology with abelian local coefficients.
		
		\medskip
		
		To prove the second part---that the isomorphism is independent of the choice of scissors embedding---let $e,e' \colon \scr{P} \to \scr{Q}$ be two scissors embeddings. Compose these embeddings with an inclusion of the form $i\colon \scr{Q} \to \scr{Q}^{\sqcup 2}$, and let $\sigma$ be the self-map of $\scr{Q}^{\sqcup 2}$ that interchanges the two copies. Then $ie$ and $\sigma ie'$ are disjoint. By \cref{disjoint_implies_conjugate}, there is an automorphism $h\colon \scr{Q}^{\sqcup 2} \simeq \scr{Q}^{\sqcup 2}$ such that $hie = \sigma i e'$. Note that $\vol(\scr{Q})>0$, so $i$ induces an isomorphism on homology. On the other hand, the homomorphisms on the group $\aut_{\category{UG}(\cA)}(\scr{Q}^{\sqcup 2})$ induced by $h$ and $\sigma$ are inner automorphisms that conjugate by $h$ and $\sigma$, respectively. Therefore they induce the identity on group homology. It follows that $e$ and $e'$ induce identical maps on group homology.
		
		For abelian local coefficients, we run the same argument, except that we arrange for $\sigma$ and $h$ to lie in the commutator subgroup. To accomplish this we let $i$ be the embedding of $\scr{Q}$ into \emph{three} copies of $\scr{Q}$, and we let $\sigma$ be the 3-cycle that rotates the copies around. Then $\sigma$ is in the commutator subgroup. For $h$, we run the proof of \cref{disjoint_implies_conjugate} but with the three disjoint embeddings $ie$, $\sigma ie'$, and $\sigma^2 ie$. Then we construct $h$ as a 3-cycle that rotates the three embedded copies of $\scr{P}$ around. This still has the property that $hie = \sigma ie'$, but now $h$ is a commutator as well.
	\end{proof}
	
	\begin{remark}
		Without property \ref{enum:ass-vol-archimedean}, the above proof still gives homological stability in a range of degrees depending on the volumes of $\scr{A}$ and $\scr{X}$. For example, there is an assembler of finite sets, with volume function given by cardinality. This satisfies properties \ref{enum:ass-vol-existence} and \ref{enum:ass-vol-zylev} but not property \ref{enum:ass-vol-archimedean}. Applying the above argument to this assembler, one recovers the homological stability theorem for symmetric groups \cite[Section 5.1]{randalwilliamswahl2017}.
	\end{remark}
	
	\subsection{Stability for S-assemblers} We would like to repeat the above argument for S-assemblers, but unfortunately the (local) cancellation condition of \cite[Definition 1.9]{randalwilliamswahl2017} need not hold. Indeed, while EAZ-assemblers satisfy \ref{enum:ass-vol-zylev} by hypothesis, S-assemblers rarely do and \ref{enum:ass-weak-zylev} will serve as a substitute. The latter is not sufficient for cancellation, but the following device from \cite{szymikwahl2019} (see \cite[Remark 1.11]{randalwilliamswahl2017}) gets around the issue.

	\begin{definition}
		The \emph{restricted scissors congruence groupoid} $\widetilde\cG(\cA)$ is the following wide subcategory of $\cG(\cA)$. There are no morphisms $(I,\{P_i\}_{i \in I}) \to (J,\{Q_j\}_{j \in J})$ unless there exists a bijection $\phi\colon I \to J$ along which $P_i \cong Q_{\phi(i)}$ in the category $\cA$, i.e.~the two sets of objects are isomorphic in $\cA$ up to rearrangement. When this condition holds, we define the morphisms to be the same as in $\cG(\cA)$, i.e.~the scissors congruences $\scr{P} \simeq \scr{Q}$. We emphasize that for each of these morphisms, the scissors congruence $\scr{P} \simeq \scr{Q}$ does not have to respect the decomposition or the bijection $\phi$ in any way.
	\end{definition}

	The category $\widetilde\cG(\cA)$ inherits the symmetric monoidal structure from $\cG(\cA)$, as the required associativity, unit, and symmetry isomorphisms of $\cG(\cA)$ all lie in the subcategory $\smash{\widetilde\cG(\cA)}$, because if we concatenate tuples of objects of $\cA$ in different orders, we end up with the same tuple up to rearrangement of the terms. The restricted scissors embedding category $\cU\widetilde{\cG}(\cA)$ is constructed as in \cref{scissors_embedding_category}, applied to the symmetric monoidal category $\widetilde{\cG}(\cA)$. Inspecting the definition, we get:
	
	\begin{definition}The restricted scissors embedding category is the wide subcategory $\cU \widetilde{\cG}(\cA)$ of $\cU\cG(\cA)$ with the following morphisms. There is only a morphism $(I,\{P_i\}_{i \in I}) \to (J,\{Q_j\}_{j \in J})$ if there is an injection $\phi\colon I \to J$ along which $P_i \cong Q_{\phi(i)}$, i.e.~the objects in the first tuple are up to isomorphism a subset of the objects in the second tuple. In that case, the morphisms are those scissors embeddings $(I,\{P_i\}_{i \in I}) \to (J,\{Q_j\}_{j \in J})$ in $\cU\cG(\cA)$ whose complement is scissors congruent to $(J \setminus \phi(I),\{Q_j\}_{j \in J \setminus \phi(I)})$. Note that this complement, up to isomorphism, does not depend on the choice of $\phi$, and other than this condition on the complement, the scissors embedding does not have to respect the map $\phi$ in any way.
	\end{definition}
	
	The subcategory $\cU\widetilde{\cG}(\cA)$ inherits most properties of $\cU\cG(\cA)$ from earlier. It has the same automorphism groups, every morphism is a monomorphism, and $\cU\widetilde{\cG}(\cA)$ is locally standard at any pair $(\scr{A},\scr{X})$ with $\scr{X} \not \simeq \varnothing$ (\cref{lem:ug-conditions}).

	Property \ref{enum:ass-vol-zylev} is automatically true in $\widetilde{\cG}(\cA)$ because the existence of an isomorphism $\scr{P} \sqcup \scr{Q} \simeq \scr{P}' \sqcup \scr{Q}$ in $\widetilde{\cG}(\cA)$ gives the extra condition that the two sides are sums of the same terms up to rearrangement. It follows that, after changing the rearrangement, it sends $\scr{Q}$ to $\scr{Q}$ and therefore the rearrangement gives an isomorphism $\scr{P} \simeq \scr{P}'$ in $\widetilde{\cG}(\cA)$. Following the proof of \cref{lem:ug-conditions-hom-eaz}, we obtain:
	
	\begin{lemma}\label{lem:ug-conditions-hom-s} If $\cA$ is an assembler, then $\cU\widetilde{\cG}(\cA)$ is homogeneous.
	\end{lemma}
	
	The upshot is that passage to the $\widetilde{\cG}(\cA)$ guarantees conditions \textbf{H1} and \textbf{H2} from \cite[Definition 1.3]{randalwilliamswahl2017} hold for any assembler $\cA$. The cost is that we have to deal with the semi-simplicial sets $\widetilde{W}_n(\scr{A},\scr{X})$ and simplicial complexes $\widetilde{S}_n(\scr{A},\scr{X})$ constructed as in \cref{def:wn} and \cref{def:sn} from $\widetilde{\cG}(\cA)$ rather than $\cG(\cA)$. 
	
	\begin{lemma}\label{lem:s-wn-connectivity} Suppose $\cA$ is an S-assembler. If $\scr{A},\scr{X} \not \simeq \varnothing$, then $\widetilde{W}_n(\scr{A},\scr{X})$ and $\widetilde{S}_n(\scr{A},\scr{X})$ are at least $(n-2)$-connected for all $n \geq 0$.\end{lemma}
	
	\begin{proof}
	Unwinding definitions, the $(n-1)$-skeleton of $\widetilde{S}_n(\scr{A},\scr{X})$ is the subcomplex of $S_n(\scr{A},\scr{X})$ where a $(p+1)$-tuple of embeddings $\scr{X} \to \scr{A} \sqcup \scr{X}^{\sqcup n}$ forms a $p$-simplex if and only if $p+1 \leq n$ and they form a scissors embedding $\scr{X}^{\sqcup p+1} \to \scr{A} \sqcup \scr{X}^{\sqcup n}$ whose complement is isomorphic to $\scr{A} \sqcup \scr{X}^{\sqcup n-p-1}$. To verify the last condition, recall that \ref{enum:ass-weak-zylev} holds by \cref{s-implies-wz}. Since $\scr{A} \not\simeq \varnothing$, then \ref{enum:ass-weak-zylev} implies that the complement is isomorphic to $\scr{A} \sqcup \scr{X}^{\sqcup n-p-1}$ if and only it is non-empty.
		
	Therefore a $(p+1)$-tuple of embeddings $\scr{X} \to \scr{A} \sqcup \scr{X}^{\sqcup n}$ forms a $p$-simplex in the $(n-1)$-skeleton of $\widetilde{S}_n(\scr{A},\scr{X})$ if and only if $p+1 \leq n$ and their complement is nonempty. In other words, the two complexes $E(\scr{A} \sqcup \scr{X}^{\sqcup n},\scr{X})$ and $\widetilde{S}_n(\scr{A},\scr{X})$ have the same $(n-1)$-skeleta. Thus $\widetilde{S}_n(\scr{A},\scr{X})$ is at least $(n-2)$-connected for all $n \geq 0$ by \cref{thm:connectivity-of-complexes-of-destabilisation} \ref{enum:connectivity-ii}. By \cref{s_to_w} for $k=1$ and $a=2$, $\widetilde{W}_n(\scr{A},\scr{X})$ is also at least $(n-2)$-connected for all $n \geq 0$.
	\end{proof}

	\begin{theorem}\label{thm:embeddings-induce-homology-isomorphisms-zae} Let $\cA$ be an S-assembler. Then for any scissors embedding $\scr{P} \hookrightarrow \scr{Q}$ with $\scr{P} \not \simeq \varnothing$ the map 
	\[\aut_{\category{UG}(\cA)}(\scr{P}) \longrightarrow \aut_{\category{UG}(\cA)}(\scr{Q})\] 
	induces an isomorphism on homology, and the same is true with abelian local coefficients. These isomorphisms are independent of the choice of scissors embedding.
	\end{theorem}

	\begin{proof}
		We first prove the inclusions $\scr{A} \to \scr{A} \sqcup \scr{X}$ induce an isomorphism on the homology of the automorphism groups when $\scr{A} \not\simeq \varnothing$. The argument proceeds as in \cref{thm:embeddings-induce-homology-isomorphisms-vol}. Without loss of generality $\scr{X} \not\simeq \varnothing$. Using property \ref{enum:ass-s}, we choose a scissors embedding $\scr{X}^{\sqcup 2q+1} \to \scr{A}$ with nonempty complement $\scr{A'} \not\simeq \varnothing$. By \cref{lem:s-wn-connectivity}, the semi-simplicial sets $\widetilde{W}_n(\scr{A'},\scr{X})$ are $(n-2)$-connected, for all $n \geq 0$. Therefore by \cite[Theorem 3.1]{randalwilliamswahl2017} the map
		\[\aut(\scr{A'} \sqcup \scr{X}^{\sqcup n}) \longrightarrow  \aut(\scr{A'} \sqcup \scr{X}^{\sqcup n + 1})\]
		is an isomorphism on $q$th homology when $q \leq \frac{n-1}{2}$. Adding the copies of $\scr{X}$ back in proves the isomorphism on homology in the desired degree $q$. The proof for abelian local coefficients proceeds in the same way with $(3q+3)$ copies of $\scr{X}$. The proof that the isomorphism is independent of the embedding is the same as in \cref{thm:embeddings-induce-homology-isomorphisms-vol}, with the condition that $\vol(\scr{Q}) > 0$ replaced by the condition that $\scr{Q} \not\simeq \varnothing$.
	\end{proof}

	\subsection{Quasi-perfection}
	Recall that a group $G$ is said to be \emph{quasi-perfect} if its commutator subgroup, denoted by $G' \coloneqq [G,G]$, is perfect; this is equivalent to $H_1(G') = 0$. In this short section, we use our homological stability results \cref{thm:embeddings-induce-homology-isomorphisms-vol} and \cref{thm:embeddings-induce-homology-isomorphisms-zae} to show that $\aut_{\category{UG}(\cA)}(\scr{P})$ is often quasi-perfect:

	\begin{theorem}
		\label{thm:quasi-perfectness} Suppose that $\cA$ is an S-assembler and $\scr{P} \not \simeq \varnothing$, or that $\cA$ is a EAZ-assembler and $\vol(\scr{P}) > 0$. Then $\aut_{\category{UG}(\cA)}(\scr{P})$ is quasi-perfect.
	\end{theorem}

	\begin{proof}
		The discussion in \cite[Section 3.2]{randalwilliamswahl2017} shows that \[\aut_{\category{UG}(\cA)}(\scr{P}^{\sqcup \infty}) \coloneqq \underset{n \to \infty}{\colim}\, \aut_{\category{UG}(\cA)}(\scr{P}^{\sqcup n})\]
		is a quasi-perfect group, hence $H_1(\aut_{\category{UG}(\cA)}(\scr{P}^{\sqcup \infty})') = 0$. So it remains to show that $H_1(\aut_{\category{UG}(\cA)}(\scr{P})') \to H_1(\aut_{\category{UG}(\cA)}(\scr{P}^{\sqcup \infty})')$ is an isomorphism.
		
		This follows from the discussion in \cite[Section 3.1]{randalwilliamswahl2017}, which we adapt to our setting: we first apply \cref{thm:embeddings-induce-homology-isomorphisms-vol} or \cref{thm:embeddings-induce-homology-isomorphisms-zae} with trivial coefficients to see that
		\[ H_1(\aut_{\category{UG}(\cA)}(\scr{P})) \longrightarrow  H_1(\aut_{\category{UG}(\cA)}(\scr{P}^{\sqcup \infty}))\]
		is an isomorphism. Then we apply \cref{thm:embeddings-induce-homology-isomorphisms-vol} or \cref{thm:embeddings-induce-homology-isomorphisms-zae} again with the abelian coefficient system $\Z[H_1(\aut_{\category{UG}(\cA)}(\scr{P}^{\sqcup \infty}))]$. Since the homology of a group $G$ with coefficients in $\Z[H_1(G)]$ is isomorphic to the homology of its commutator subgroup $G'$ by the Shapiro lemma, this yields the isomorphism $H_1(\aut_{\category{UG}(\cA)}(\scr{P})') \cong H_1(\aut_{\category{UG}(\cA)}(\scr{P}^{\sqcup \infty})') = 0$.
	\end{proof}

	\section{Connection to assembler K-theory}
	\label{sec:connection-to-assembler-k-theory}
	
	In this section we establish how the scissors automorphism groups are related to assembler K-theory, both via group completion and the plus construction.

	\subsection{Assembler K-theory as group completion}
	We first establish a ``group-completion'' description of the algebraic K-theory $K(\cA)$ of an assembler $\cA$. This is implicit in \cite{zakharevich2014} and is presumably known to the experts.
	
	The assembler K-theory spectrum $K(\cA)$ of \cite[Definition 2.12]{zakharevich2014} is defined to be the Segal K-theory of the symmetric monoidal category $\cW^\circ(\cA)$ appearing in \cref{empty_sets_in_W_discussion}. The underlying infinite loop space can be described as follows: the symmetric monoidal structure on $\cW^\circ(\cA)$ induces an $E_\infty$-space structure on $B\cW^\circ(\cA)$, and we group-complete this by applying $\Omega B(-)$:
	\[\Omega^\infty K(\cA) \coloneqq \Omega B(B\cW^\circ(\cA)).\]

	\begin{theorem}\label{thm:group-completion}
		For any assembler $\category{A}$, we have an equivalence
		\[\Omega^\infty K(\cA) \simeq \Omega B(B\cG(\cA)).\]
	\end{theorem}

	\begin{proof}
		It suffices to prove that the maps
		\[B\cW^\circ(\cA) \longrightarrow B\cW(\cA) \longrightarrow  B\cG(\cA)\]
		are equivalences of $E_\infty$-spaces. The left map is an equivalence of $E_\infty$-spaces because it is induced by the inclusion $i^\circ \colon \cW^\circ(\cA) \to \cW(\cA)$, which is symmetric monoidal and admits a right adjoint. The right map is a map of $E_\infty$-spaces because it is induced by the localisation functor $\cW(\cA) \to \cW(\cA)[\cW(\cA)^{-1}] = \cG(\cA)$, which is symmetric monoidal by \cref{lem:loc-sym} and \cref{lem:covers-loc}. Next, recall that inverting all of the morphisms in a category does not in general induce an equivalence on classifying spaces. However, any localisation satisfying the conditions of \cref{def:loc-conditions} does induce an equivalence on classifying spaces, by combining \cite[Section ~4.3]{DwyerKanI} with \cite[2.2.~Proposition and 7.2.~Proposition]{DwyerKanII}. The conditions are satisfied here by \cref{lem:covers-loc}, so the right map above is an equivalence as well.
	\end{proof}

	\begin{remark}Picking a skeleton, we may identify $B\cG(\cA)$ as $\bigsqcup_{[P]} B\aut_{\cG(\cA)}(P)$ as in the introduction, where the coproduct is indexed by isomorphism classes of objects. This description is useful, even though it obscures the $E_\infty$-space structure.\end{remark}

	\subsection{Assembler K-theory as stable homology} Now that we have described assembler K-theory as a group completion, we can use scissors automorphisms to give a presentation of $K_1$, and conversely use K-theory to describe the stable homology of scissors automorphism groups.
	
	\begin{corollary}\label{cor:stable-homology}For any assembler $\cA$, we have an isomorphism
		\[H_*(B\cG(\cA))[\pi_0^{-1}] \overset{\cong}\longrightarrow H_*(\Omega^\infty K(\cA))\]
	and the left side can be computed as a filtered colimit of homology groups $H_*(B\aut_{\cG(\cA)}(\scr{P}))$.
	\end{corollary}

	\begin{proof}The map of $E_\infty$-spaces $B\cG(\cA) \to \Omega B(B\cG(\cA))$ 	induces by the group completion theorem \cite[Proposition 1]{mcduffsegal1976} an isomorphism
		\[ H_*(B\cG(\cA))[\pi_0^{-1}] \overset{\cong}\longrightarrow H_*(\Omega B(B\cG(\cA))),\]
	and we can use \cref{thm:group-completion} to rewrite the right side. (To apply this result, we remember only that $B\cG(\cA)$ is a homotopy-commutative $E_1$-space, and we strictify it to a homotopy-commutative topological monoid $M$.)
	
	A standard consequence of this group completion theorem is that if we restrict to a single connected component, the homology of $\Omega_0 B(B\cG(\cA))$ is the filtered colimit of the homology of the components,
		\[ H_*(\Omega_0 B(B\cG(\cA))) = \underset{[\scr{P}] \in \pi_0(B\cG(\cA))}\colim \, H_*(B\aut_{\cG(\cA)}(\scr{P})). \]
		The colimit system has a map $[\scr{P}] \to [\scr{P} \sqcup \scr{Q}]$ for each $[\scr{P}], [\scr{Q}] \in \pi_0(B\cG(\cA))$, which applies $(-) \sqcup \id_\scr{Q}$ and then any isomorphism in $\cG(\cA)$ between $\scr{P} \sqcup \scr{Q}$ and the chosen element of $\cG(\cA)$ representing the isomorphism class $[\scr{P} \sqcup \scr{Q}]$. (In forming the diagram, we have implicitly chosen one representative $\scr{P}$ from each isomorphism class $[\scr{P}]$.)
	\end{proof}
	
	\begin{proposition}\label{complete_presentation}
		For any assembler $\cA$, the partial presentation of $K_1(\cA)$ in \cite[Theorem B]{zak_k1} is in fact a presentation, i.e.~no additional relations are needed.
	\end{proposition}
		That is, $K_1(\cA)$ is generated as an abelian group by zigzags in $\cW^\circ(\cA)$
		\[[\scr{P} \overset{f_1}\longleftarrow \scr{R} \overset{f_2}\longrightarrow \scr{P}],\]
		subject to the relations:
		\begin{enumerate}[label={(\Alph*)}]
			\item \label{enum:zpres-a} $[\scr{P} \overset{f}\longleftarrow \scr{R} \overset{f}\longrightarrow \scr{P}] = 0$.
			\item \label{enum:zpres-b} $[\scr{P} \overset{f_1}\longleftarrow \scr{R} \overset{f_2}\longrightarrow \scr{P}]+[\scr{Q} \overset{g_1}\longleftarrow \scr{P} \overset{g_2}\longrightarrow \scr{Q}] = [\scr{Q} \overset{g_1f_1}\longleftarrow \scr{R} \overset{g_2f_2}\longrightarrow \scr{Q}]$.
			\item \label{enum:zpres-c} $[\scr{P} \overset{f_1}\longleftarrow \scr{R} \overset{f_2}\longrightarrow \scr{P}]+[\scr{Q} \overset{g_1}\longleftarrow \scr{S} \overset{g_2}\longrightarrow \scr{Q}] = [\scr{P} \sqcup \scr{Q} \xleftarrow{f_1\sqcup g_1} \scr{R} \sqcup \scr{S} \xrightarrow{f_2 \sqcup g_2} \scr{P} \sqcup \scr{Q}]$.
		\end{enumerate}
	
	\begin{proof} We use the description of $K_1(\cA) \cong H_1(\Omega^\infty_0 K(\cA))$ as a filtered colimit from \cref{cor:stable-homology}. We get a generator $g \in K_1(\cA)$ for each object $\scr{P}$ and element of $H_1(\aut_{\cG(\cA)}(\scr{P})) = \smash{\aut_{\cG(\cA)}(\scr{P})^{\ab}}$. Since elements of $\aut_{\cG(\cA)}(\scr{P})$ are represented by zigzags $\scr{P} \leftarrow \scr{R} \rightarrow \scr{P}$, we could also say we have a generator for each such zigzag. Then the following relations are sufficient to give the colimit:
		\begin{enumerate}[label={(\arabic*)}]
			\item \label{enum:pres-1} Any refinement of the zig-zag must give the same element of $K_1(\cA)$. In other words, the resulting map $\aut_{\cG(\cA)}(\scr{P}) \to K_1(\cA)$ is well-defined as a map of sets.
			\item \label{enum:pres-2} Composition of scissors automorphisms must go to addition in $K_1(\cA)$. In other words, the resulting map $\aut_{\cG(\cA)}(\scr{P}) \to K_1(\cA)$ is a homomorphism.
			\item \label{enum:pres-3} Any scissors congruence $\scr{P} \congto \scr{Q}$ must induce a commuting diagram
			\[ \begin{tikzcd}
				\aut_{\cG(\cA)}(\scr{P}) \ar[rr] \ar[rd] && \aut_{\cG(\cA)}(\scr{Q}) \ar[dl] \\[-6pt]
				& K_1(\cA).
			\end{tikzcd} \]
			\item \label{enum:pres-4} For any $\scr{P}$ and $\scr{Q}$, extending by the identity must induce a commuting diagram
			\[ \begin{tikzcd}
				\aut_{\cG(\cA)}(\scr{P}) \ar[rr] \ar[rd] && \aut_{\cG(\cA)}(\scr{P} \sqcup Q) \ar[dl] \\[-6pt]
				& K_1(\cA).
			\end{tikzcd} \]
		\end{enumerate}
	Now we compare to Zakharevich's presentation, which has the same generators, one for each zigzag $\scr{P} \leftarrow \scr{R} \rightarrow \scr{P}$. In \cite{zak_k1} this generator corresponds to the class in $K_1(\cA)$ that traverses the loop in the classifying space $B\cW^\circ(\cA)$ formed by the two morphisms $\scr{R} \rightrightarrows \scr{P}$. This agrees with the $K_1$ class given by this generator in our presentation.
	
	Let us assume the relations \ref{enum:zpres-a}, \ref{enum:zpres-b}, and \ref{enum:zpres-c} from Zakharevich's presentation. Then it is a pleasant exercise to deduce the relations \ref{enum:pres-1}, \ref{enum:pres-3}, and \ref{enum:pres-4}. It is much less straightforward to deduce \ref{enum:pres-2}, so we spell out the proof for this, following the argument from \cite[\S 4]{zak_k1}.
	
	Suppose we are given two scissors automorphisms $[\scr{P} \overset{f_1}\longleftarrow \scr{R} \overset{f_2}\longrightarrow \scr{P}]$ and $[\scr{P} \overset{f_1'}\longleftarrow \scr{R'} \overset{f_2'}\longrightarrow \scr{P}]$. Pick a common refinement $\scr{R}''$ of $\scr{R}$ and $\scr{R}'$ as well as a common refinement $\scr{R}'''$ of $\scr{R}'$ and $\scr{R}''$. Then the composition of the two scissors automorphisms is given by $[\scr{P} \overset{f_1 \circ g_1}\longleftarrow \scr{R}'' \overset{f_2' \circ g_2}\longrightarrow \scr{P}]$ and sits inside the following commutative diagram:
	\[
		\begin{tikzcd}
			\scr{R}''' \arrow{d}[description]{h_1} \arrow{r}[description]{h_2} &[5pt] \scr{R}'' \arrow{d}[description]{g_1} \arrow{rd}[description]{g_2} &[5pt] &[5pt] \\[3pt]
			\scr{R}' \arrow{d}[description]{f_2'} \arrow{dr}[description]{f_1'} & \scr{R} \arrow{d}[description]{f_1} \arrow{dr}[description]{f_2} & \scr{R}' \arrow{dr}[description]{f_2'} \arrow{d}[description]{f_1'} & \\[3pt]
			\scr{P} & \scr{P} & \scr{P} & \scr{P}
		\end{tikzcd}
	\]
	Then as in \cite[\S 4]{zak_k1} we deduce \ref{enum:pres-2} from the relations \ref{enum:zpres-a} and \ref{enum:zpres-b}:
	\begin{align*}
		&[\scr{P} \xleftarrow{f_1 \circ g_1} \scr{R}'' \xrightarrow{f_2' \circ g_2} \scr{P}]\\
		&= [\scr{P} \xleftarrow{f_1 \circ g_1 \circ h_2} \scr{R}''' \xrightarrow{f_2' \circ g_2 \circ h_2} \scr{P}] \tag{by \ref{enum:zpres-a} and \ref{enum:zpres-b}}\\
		&= [\scr{P} \xleftarrow{f_1' \circ h_1} \scr{R}''' \xrightarrow{f_2' \circ g_2 \circ h_2} \scr{P}] \tag{by commutativity}\\
		&= [\scr{R}' \overset{h_1} \longleftarrow \scr{R}''' \xrightarrow{g_2 \circ h_2} \scr{R}'] + [\scr{P} \overset{f_1'}\longleftarrow \scr{R'} \overset{f_2'}\longrightarrow \scr{P}] \tag{by \ref{enum:zpres-b}}\\
		&= [\scr{P} \xleftarrow{f_1' \circ h_1} \scr{R}''' \xrightarrow{f_1' \circ g_2 \circ h_2} \scr{P}] + [\scr{P} \overset{f_1'}\longleftarrow \scr{R'} \overset{f_2'}\longrightarrow \scr{P}] \tag{by \ref{enum:zpres-a} and \ref{enum:zpres-b}}\\
		&= [\scr{P} \xleftarrow{f_1 \circ g_1 \circ h_2} \scr{R}''' \xrightarrow{f_2 \circ g_1 \circ h_2} \scr{P}] + [\scr{P} \overset{f_1'}\longleftarrow \scr{R'} \overset{f_2'}\longrightarrow \scr{P}] \tag{by commutativity}\\
		&= [\scr{P} \overset{f_1}\longleftarrow \scr{R} \overset{f_2}\longrightarrow \scr{P}] + [\scr{P} \overset{f_1'}\longleftarrow \scr{R'} \overset{f_2'}\longrightarrow \scr{P}] \tag{by \ref{enum:zpres-a} and \ref{enum:zpres-b}}
	\end{align*}
	
	Since we know that the relations \ref{enum:pres-1}-\ref{enum:pres-4} are sufficient, we conclude that the relations \ref{enum:zpres-a}, \ref{enum:zpres-b}, and \ref{enum:zpres-c} from Zakharevich's presentation are also sufficient.
	\end{proof}

	In the special case of an EAZ- or S-assembler, we get the following stronger result, involving only the scissors automorphism group of a single object. Recall that a map $f \colon X \to Y$ is \emph{acyclic} if it induces an isomorphism $H_*(X;f^*\cL) \to H_*(Y;\cL)$ for any local coefficient system $\cL$ on $Y$.
	
	\begin{corollary}\label{cor:stable_homology_acyclic}
		Suppose that $\cA$ is an S-assembler and $P \not \simeq \varnothing$, or $\cA$ is an EAZ-assembler and $\vol(P) > 0$. Then the map
		\[ B\aut_{\cG(\cA)}(P) \longrightarrow \Omega^\infty_{[P]} K(\cA) \]
		is acyclic, where the subscript $[P]$ denotes the path component of $P$.
	\end{corollary}

	\begin{proof}
		We use the improvement of the group-completion theorem from \cite[Theorem 1.1]{RandalWilliamsGC}. For any object $Q$ of $\cA$, there exists an $n \geq 0$ such that there is a scissors embedding $Q \hookrightarrow P^{\sqcup n}$, so the hypothesis in Randal-Williams's theorem is satisfied. Thinking of $B\cG(\cA)$ as a homotopy-commutative $E_1$-space, we strictify it to a homotopy-commutative topological monoid $M$, and we form the mapping telescope 
		\[M_\infty \coloneqq \hocolim (M \overset{P \sqcup -}\longrightarrow M  \overset{P \sqcup -}\longrightarrow \cdots).\]
		Then \cite[Theorem 1.1]{RandalWilliamsGC} says that there is an acyclic map
		\[ p \colon M_\infty \longrightarrow \Omega BM \]
		extending the canonical map $M \to \Omega BM$. 
	
		As in the proof of \cref{complete_presentation}, we restrict to one component of $M_\infty$ and $\Omega BM$, this time the component corresponding to the element $[P]$. Unlike that proof, we consider not just ordinary homology but also homology with local coefficients coming from $\Omega BM$. Note these are all abelian since the components of $\Omega BM$ have abelian fundamental group.
		
		For any (necessarily abelian) local coefficient system $\cL$ on $\Omega B(B\cG(\cA))$, the map $p$ induces an isomorphism
		\[ \underset{n\to\infty}\colim \, H_*(B\aut_{\cG(\cA)}(P^{\sqcup n});p^* \cL) \overset{\cong}\longrightarrow H_*(\Omega_{[P]} B(B\cG(\cA));\cL). \]
		However, by the assumption on $P$ and either \cref{thm:embeddings-induce-homology-isomorphisms-vol} or \cref{thm:embeddings-induce-homology-isomorphisms-zae}, the maps in the colimit system on the left are all isomorphisms. We conclude that the canonical map gives an isomorphism
		\[ H_*(B\aut_{\cG(\cA)}(P);p^* \cL) \cong H_*(\Omega_{[P]} B(B\cG(\cA));\cL). \]
		Finally, we rewrite $\Omega_{[P]} B(B\cG(\cA))$ as $\Omega^\infty_{[P]} K(\cA)$ using \cref{thm:group-completion}.
	\end{proof}

	\subsection{Assembler K-theory as plus construction} Finally, we deduce from this a ``plus-construction'' description of the algebraic K-theory of any EAZ- or S-assembler $\cA$. We refer the reader to \cite[Chapter~IV, Section~1]{weibel2013} for details about Quillen's plus-construction.
	
	\begin{theorem}\label{thm:plus-construction}
		Suppose that $\cA$ is an S-assembler and $P \not \simeq \varnothing$, or that $\cA$ is an EAZ-assembler and $\vol(P) > 0$. Then there is an equivalence
		\[ K_0(\cA) \times B\aut_{\cG(\cA)}(P)^+ \simeq \Omega^\infty K(\cA), \]
		where $(-)^+$ denotes the plus-construction with respect to the maximal perfect subgroup.
	\end{theorem}

	\begin{proof}
		Both sides have the same collection of path components and the property that all path components are equivalent, so it suffices to establish an equivalence on a single path component.
		
		The map $B\aut_{\cG(\cA)}(P) \to \Omega^\infty_{[P]} K(\cA)$ is acyclic by \cref{cor:stable_homology_acyclic}. On fundamental groups it therefore induces a group homomorphism that is an isomorphism after abelianisation. Since the target space is a component of a loop space, its fundamental group is abelian. The map on fundamental groups hence identifies with the quotient map $\aut_{\cG(\cA)}(P) \to \aut_{\cG(\cA)}(P)^{ab}$. This homomorphism is surjective and has kernel the commutator subgroup of $\aut_{\cG(\cA)}(P)$, which is perfect by \cref{thm:quasi-perfectness}. By the uniqueness of the plus-construction \cite[Chapter IV, Theorem 1.5, Item (3)]{weibel2013}, we conclude that the induced map
		\[B\aut_{\cG(\cA)}(P)^+ \longrightarrow \Omega^\infty_{[P]} K(\cA) \]
		is an equivalence.
	\end{proof}

	\section{Classical scissors congruence} \label{sec:applications-i-classical-sissors-congruence}
	Our motivating examples come from classical scissors congruence in Euclidean, hyperbolic, or spherical geometry. In this section we recall how these arise from assemblers (\cref{sec:classical-assemblers}) and verify these are EAZ- or S-assemblers under mild assumptions (\cref{sec:ea}). Therefore our homological stability results apply to classical scissors congruence. In the sequel paper \cite{KLMMS-3} we use this to compute the homology $H_*(\aut(P))$ in several cases of interest.
 
	\subsection{Assemblers for classical scissors congruence} \label{sec:classical-assemblers} 
	We now explain how scissors congruence of $n$-dimensional polytopes in Euclidean, hyperbolic, or spherical geometry arises from an assembler; this was first done in \cite[Section 3]{zakharevich2012} and \cite[Example 3.6]{zakharevich2014}, see also \cite[Section 5.2]{zakharevich2014}, \cite[Example 2.9]{bgmmz}, and \cite[Section 2]{malkiewich2022}.

	One starts with one of the following three subspaces $X^n \subseteq \mathbb{R}^{n+1}$ for $n \geq 0$, referred to as \emph{geometries}:
	\begin{enumerate}
		\item $H^n = \{(x_0,\ldots,x_n) \mid -x_0^2+\sum_{i=1}^n x_i^2 = -1 \text{ and }x_0>0\}$,
		\item $E^n = \{(x_0,\ldots,x_n) \mid x_0 = 1\}$,
		\item $S^n = \{(x_0,\ldots,x_n) \mid \sum_{i=0}^n x_i^2 = 1\}$.
	\end{enumerate}
	In one of these geometries we consider the following subsets:
	\begin{itemize}
		\item A \emph{geometric subspace} of $X^n$ is the intersection of $X^n$ with a linear subspace in $\R^{n+1}$.
		\item A \emph{geometric $n$-simplex} is given by taking the convex hull of $(n+1)$ points in $\mathbb{R}^{n+1}$ in general position, and projecting the result onto $X^n$ from the origin. In the spherical case we demand that the convex hull does not contain the origin. In particular, every geometric $n$-simplex is a convex $n$-polytope.
		\item An \emph{$n$-polytope} is then a subset $P \subseteq X^n$ that is a finite union of geometric $n$-simplices.
	\end{itemize}
	We will later need the following definition:
	\begin{itemize}
		\item A \emph{convex $n$-polytope} in $X^n$ is an $n$-polytope obtained as an intersection of $X$ with finitely many linear half-spaces in $\R^{n+1}$. In the case of $S^n$, we also require that the polytope lies in an open hemisphere before it is considered to be convex.
	\end{itemize}
	As a consequence of \cite[Proposition 2.2]{malkiewich2022}, an intersection of $X^n$ with finitely many linear half-spaces in $\R^{n+1}$ is an $n$-polytope if and only if it has nonempty interior in $X^n$ and is bounded.

	Each of the geometries has an associated \emph{isometry group} $I(X^n) \leq \mathrm{GL}_{n+1}(\mathbb{R})$:
	\begin{enumerate}
		\item $I(H^n)$ consists of linear maps preserving the form $-x_0^2+\sum_{i=1}^n x_i^2$ and the sign of $x_0$. It is an index-two subgroup of the indefinite orthogonal group $O(1,n)$.
		\item $I(E^n)$ consists of linear maps preserving the form $\sum_{i=1}^n x_i^2$ and function $x_0$. It is a semidirect product
		\[ I(E^n) \cong \R^n \rtimes \mathrm{O}(n). \]
		\item $I(S^n) = O(n+1)$ consists of linear maps preserving the form $\sum_{i=0}^n x_i^2$.
	\end{enumerate}
	In Euclidean space, we also consider the larger \emph{affine transformation group}
	\[ A(E^n) \cong \R^n \rtimes \mathrm{GL}_n(\R)\]
	consisting of all affine-linear maps, or equivalently the subgroup of $\mathrm{GL}_{n+1}(\mathbb{R})$ preserving the function $x_0$.

	\begin{definition}\label{polytope_assembler}\,
		\begin{itemize}
			\item The category $\cX^n$ has objects given by $n$-polytopes $P$ in $X^n$. The morphisms $P \to Q$ with nonempty domain are given by isometries $g \in I(X^n)$ such that $g(P) \subseteq Q$, and if $P$ is empty there is a unique morphism $P \to Q$.
			\item More generally, for any subgroup $G \leq I(X^n)$, or $G \leq A(E^n)$ in the Euclidean case, we let $\cX^n_G$ be the category with the same objects of $\cX^n$ and morphisms given by transformations in $G$.
		\end{itemize}
	\end{definition}
	
	\begin{notation}In the Euclidean, hyperbolic, and spherical cases, we sometimes denote $\cX^n$ by $\cE^n$, $\cH^n$, or $\category{S}^n$, respectively, and similarly for $\cX^n_G$.\end{notation} 

	We endow $\cX^n_G$ with a Grothendieck topology by saying that a collection of morphisms is a covering family if the union of the images is equal to the target. Because the category only contains $n$-polytopes and not any $(n-1)$-polytopes (except the empty one), the morphisms in $\cX^n_G$ are disjoint if and only if the overlaps between their images do not contain any $n$-polytopes. This is equivalent to asking for the images of their interiors to be disjoint, or for the overlaps to have measure zero. This satisfies properties \ref{enum:assembler-initial}, \ref{enum:assembler-refinement}, \ref{enum:assembler-mono} of \cref{def:assembler}, and hence $\cX^n_G$ is a assembler.

	\begin{remark}
	The assembler $\category{G}_n$ of \cite[Section 5.2]{zakharevich2014} is equivalent to $\cE^n$ as long as $n>0$.
	\end{remark}

	We therefore get scissors automorphism groups that we abbreviate as
	\[\aut_G(P) \coloneqq \aut_{\category{UG}(\cX_G^n)}(P), \]
	for any of the above geometries, any group of transformations $G$, and any polytope $P \subseteq X$. For $\cX^n = \cE^n$, so in particular taking $G = I(E^n)$, this recovers the scissors automorphism groups from the introduction.
	
	\subsection{Proving the EAZ and S axioms}\label{sec:ea}
	To apply our homological stability results to the groups $\aut_G(P)$, we have to verify that $\cX^n_G$ is an EAZ- or S-assembler. This seems to require a few assumptions about $G$, so we prove the assumption in different cases separately. Combing our previous results with those established in the following subsections, we obtain:
	
	\begin{corollary}\label{cor:scissors-aut-geometries}
		In each of the subsequent cases, the scissors automorphism groups $\aut_G(P) = \aut_{\category{UG}(\cX_G^n)}(P)$ have homology groups that are independent of $P$ as long as $P$ is nonempty, given by
		\[ H_*(\aut_G(P)) \cong H_*(\Omega^\infty_0 K(\cX_G^n)). \]
	\end{corollary}
	
	\subsubsection{Euclidean geometry with isometries} 
	We begin with the case of Euclidean geometry and $G$ a subgroup of the isometries containing enough translations.

	\begin{proposition}\label{euclidean_ea}
	For $n>0$, the assembler $\cE^n_G$ is an EAZ-assembler with volume function given by the usual Euclidean volume, provided $G \leq I(E^n)$ and $G$ contains a group of translations that is dense in $\realnumbers^n$.
	\end{proposition}

	\begin{proof}
	Since $G$ consists of isometries, the Euclidean volume is additive over covers, so it is indeed a volume function (\cref{volume_function}). It remains to verify properties \ref{enum:ass-vol-existence}, \ref{enum:ass-vol-archimedean} and \ref{enum:ass-vol-zylev}.
	
	Property \ref{enum:ass-vol-archimedean} is straightforward and follows by intersecting a polytope in $E^n$ with a sufficiently fine triangulation of $E^n$. Property \ref{enum:ass-vol-zylev} is a consequence of \cref{eaz_reduction} once we have checked \ref{enum:ass-vol-existence} and \ref{enum:ass-vol-iso-0}. However, \ref{enum:ass-vol-iso-0} holds by definition, because the only $n$-polytope $P$ in $\cE^n_G$ with $\vol(P) = 0$ is the empty $n$-polytope. So, it remains to verify \ref{enum:ass-vol-existence}.
	
	A proof of property \ref{enum:ass-vol-existence} for the case $G = I(E^n)$ can be found in \cite{sah1979}: ($\ast$) on p.31 loc.cit.~is the Euclidean case. Before giving an argument that applies in our slightly more general setting, we discuss the argument for \ref{enum:ass-vol-existence} when $G$ contains the full translation group.
	
	Note that every formal disjoint union of polytopes $\scr{P} \in \cG(\cE^n_G)$ can be represented by an actual polytope in $n$-Euclidean space $E^n$, so it suffices to consider two polytopes $P$ and $Q$ in $E^n$ with $\vol(P) < \vol(Q)$. Since polytopes are Jordan measurable (i.e.~their indicator functions are Riemann integrable), the volume of $P$ can be approximated arbitrarily well by small cubes $C_i$ that cover $P$, and the volume of $Q$ can be approximated arbitrarily well by small cubes $D_j$ that are completely contained in $Q$.

	\vspace{1em}
	\centerline{
	\def\svgwidth{3.0in}
\begingroup%
  \makeatletter%
  \providecommand\color[2][]{%
    \errmessage{(Inkscape) Color is used for the text in Inkscape, but the package 'color.sty' is not loaded}%
    \renewcommand\color[2][]{}%
  }%
  \providecommand\transparent[1]{%
    \errmessage{(Inkscape) Transparency is used (non-zero) for the text in Inkscape, but the package 'transparent.sty' is not loaded}%
    \renewcommand\transparent[1]{}%
  }%
  \providecommand\rotatebox[2]{#2}%
  \newcommand*\fsize{\dimexpr\f@size pt\relax}%
  \newcommand*\lineheight[1]{\fontsize{\fsize}{#1\fsize}\selectfont}%
  \ifx\svgwidth\undefined%
    \setlength{\unitlength}{184.78468489bp}%
    \ifx\svgscale\undefined%
      \relax%
    \else%
      \setlength{\unitlength}{\unitlength * \real{\svgscale}}%
    \fi%
  \else%
    \setlength{\unitlength}{\svgwidth}%
  \fi%
  \global\let\svgwidth\undefined%
  \global\let\svgscale\undefined%
  \makeatother%
  \begin{picture}(1,0.47619098)%
    \lineheight{1}%
    \setlength\tabcolsep{0pt}%
    \put(0,0){\includegraphics[width=\unitlength,page=1]{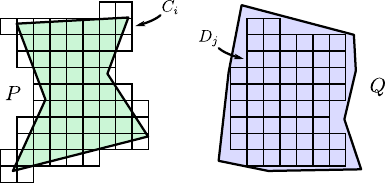}}%
  \end{picture}%
\endgroup%

	}
	\vspace{1em}
	
	We can arrange so all cubes have the same size and the number $M$ of cubes $\{C_i\}_{i \in \{1, \dots, M\}}$ containing $P$ is less than the number $N$ of cubes $\{D_i\}_{i \in \{1, \dots, N\}}$ contained in $Q$. Once this is done, we cut $P$ into the intersections $P \cap C_i$ and send each these into the corresponding cubical piece $D_i$ of $Q$ using the unique translation that maps $C_i$ to $D_i$. This gives the desired scissors embedding $P \to Q$.

	This argument can be modified to only use translations in a dense subset of $\realnumbers^n$, as follows. We approximate $Q$ by cubes $D_i$ of the same size that have small gaps in between; this is possible because the error introduced by expanding the cubes by a factor of $(1+\epsilon)$ is a factor of $(1+\epsilon)^n$, so by making $\epsilon$ small we can still get as close to the volume of $Q$ as desired. To map $P \cap C_i$ into $Q$, we take the unique translation sending $C_i$ to $D_i$ and perturb it by less than $\epsilon/2$ in each direction to a translation that lies in the subgroup $G$. This accomplishes the scissors embedding $P \to Q$ using only translations in $G$.
	\end{proof}

	\subsubsection{Euclidean geometry with scalings} \label{sec:eucl-scaling}
	Next we consider the case of Euclidean geometry where we allow non-trivial scalings in the following sense. Suppose that $G \leq A(E^n)$ contains an element that fixes a point $x \in E^n$, and if we use that point as an origin and consider the resulting linear map $T\colon \R^n \to \R^n$, then $T$ shrinks Euclidean distances, in other words $\|T\| < 1$. For instance, this is satisfied if $G$ contains a nontrivial dilatation about the point $x$.
	
	\begin{proposition}\label{euclidean_s}
	For $n>0$, the assembler $\cE^n_G$ is an S-assembler, provided $G$ contains an element that shrinks Euclidean distance and $G$ contains a group of translations that is dense in $\realnumbers^n$.
	\end{proposition}

	\begin{proof}
		We verify property \ref{enum:ass-s}. Again without loss of generality $P$ and $Q$ are nonempty polytopes, not just formal finite disjoint unions of $n$-polytopes. The first condition on $G$ implies that we can shrink $P$ to fit into a ball of arbitrarily small radius. This ball may then be translated into the interior of $Q$, since $G$ contains translations arbitrarily close to a given translation. This embeds $P$ into $Q$.
	\end{proof}
	
	\subsubsection{Hyperbolic and spherical geometries}
	Finally, we consider the hyperbolic and spherical cases. These require much more work than the Euclidean case to show that the assembler satisfies \ref{enum:ass-vol-existence}, \ref{enum:ass-vol-archimedean} and \ref{enum:ass-vol-zylev}. We focus on the hyperbolic case, since the spherical case is similar.
	
	Our proof will make significant use of first-order approximations, in other words the constant and linear terms of the Taylor series of a differentiable function. Suppose $f$ and $g$ are nonnegative differentiable real-valued functions on $[0,a)$ for some $a > 0$. We write $f(x) \sim g(x)$ if $f$ and $g$ have the same first-order approximation at $x = 0$. We write $f(x) \lesssim g(x)$ if the first-order approximation to $f$ at $x = 0$ is no larger than the first-order approximation to $g$ at $x = 0$.

	We first consider the volume of a region in any complete $n$-dimensional Riemannian manifold $M$ determined by taking a smoothly embedded disc $D^{n-1} \subseteq M$, and moving along geodesics perpendicular to $D^{n-1}$ to form a cylinder $D^{n-1} \times I \to M$, the $I$ direction parametrized by arc length.
	
	\vspace{1em}
	\centerline{
	\def\svgwidth{1.4in}
\begingroup%
  \makeatletter%
  \providecommand\color[2][]{%
    \errmessage{(Inkscape) Color is used for the text in Inkscape, but the package 'color.sty' is not loaded}%
    \renewcommand\color[2][]{}%
  }%
  \providecommand\transparent[1]{%
    \errmessage{(Inkscape) Transparency is used (non-zero) for the text in Inkscape, but the package 'transparent.sty' is not loaded}%
    \renewcommand\transparent[1]{}%
  }%
  \providecommand\rotatebox[2]{#2}%
  \newcommand*\fsize{\dimexpr\f@size pt\relax}%
  \newcommand*\lineheight[1]{\fontsize{\fsize}{#1\fsize}\selectfont}%
  \ifx\svgwidth\undefined%
    \setlength{\unitlength}{105.69598812bp}%
    \ifx\svgscale\undefined%
      \relax%
    \else%
      \setlength{\unitlength}{\unitlength * \real{\svgscale}}%
    \fi%
  \else%
    \setlength{\unitlength}{\svgwidth}%
  \fi%
  \global\let\svgwidth\undefined%
  \global\let\svgscale\undefined%
  \makeatother%
  \begin{picture}(1,1.24574229)%
    \lineheight{1}%
    \setlength\tabcolsep{0pt}%
    \put(0,0){\includegraphics[width=\unitlength,page=1]{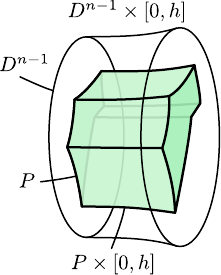}}%
  \end{picture}%
\endgroup%

	}
	\vspace{1em}
	
	\begin{lemma}\label{cylinder_volume}
		If $P \subseteq D^{n-1}$ is (Lebesgue) measurable then the volume of the image of $P \times [0,h]$ under this map is, to first order in $h$,
		\[ \vol(P \times [0,h]) \sim h\vol(P). \]
	\end{lemma}
	
	\begin{proof}
		When $h$ is sufficiently small, the cylinder $P \times [0,h]$ is embedded, so its volume may be computed as an integral
		\[ \vol(P \times [0,h]) =  \int_{P \times [0,h]} \sqrt{\det G} = \int_{D^{n-1} \times [0,h]} \chi_P \cdot \sqrt{\det G}, \]
		where $\chi_P$ is the indicator function of the subset $P \subseteq D^{n-1}$, and $G$ is the matrix of the metric tensor of $M$, pulled back to $D^{n-1} \times [0,h]$. We rewrite this as
		\[ \vol(P \times [0,h]) =  \int_0^h f(t)\ dt, \qquad f(t) = \int_{P \times \{t\}} \sqrt{\det G}. \]
		We note that $f(t)$ is continuous in $t$, since $P$ has finite measure and $\sqrt{\det G}$ is uniformly continuous on $D^{n-1} \times [0,1]$.
		
		We next show that $f(0) = \vol(P)$. Since the geodesics are perpendicular to $D^{n-1}$ and parametrized by arc length, along $D^{n-1} \times \{0\}$ the matrix $G$ has block form
		\[ G = \begin{pmatrix}
			A & 0 \\ 0 & 1
		\end{pmatrix}, \]
		where $A$ is an $(n-1)\times(n-1)$ matrix given by the metric tensor in the direction of $D^{n-1} \times \{0\}$. Therefore
		\[ f(0) = \int_{P \times \{0\}} \sqrt{\det G} = \int_{P \times \{0\}} \sqrt{\det A} = \vol(P). \]
		
		In summary, the desired volume is of the form 
		\[ \vol(P \times [0,h]) = \int_0^h f(t)\ dt, \qquad f \textup{ continuous}, \qquad f(0) = \vol(P). \]
		The result now follows from the fundamental theorem of calculus.
	\end{proof}
	
	For a hyperbolic polytope $P$ and $\epsilon > 0$, let $P_{+\epsilon} \supseteq P$ be the closed set of all points that are $\leq \epsilon$ from $P$, and similarly let $P_{-\epsilon} \subseteq P$ be the closed set of points that are $\geq \epsilon$ away from the complement of $P$.
	
	\vspace{1em}
	\centerline{
	\def\svgwidth{2.6in}
\begingroup%
  \makeatletter%
  \providecommand\color[2][]{%
    \errmessage{(Inkscape) Color is used for the text in Inkscape, but the package 'color.sty' is not loaded}%
    \renewcommand\color[2][]{}%
  }%
  \providecommand\transparent[1]{%
    \errmessage{(Inkscape) Transparency is used (non-zero) for the text in Inkscape, but the package 'transparent.sty' is not loaded}%
    \renewcommand\transparent[1]{}%
  }%
  \providecommand\rotatebox[2]{#2}%
  \newcommand*\fsize{\dimexpr\f@size pt\relax}%
  \newcommand*\lineheight[1]{\fontsize{\fsize}{#1\fsize}\selectfont}%
  \ifx\svgwidth\undefined%
    \setlength{\unitlength}{148.51714982bp}%
    \ifx\svgscale\undefined%
      \relax%
    \else%
      \setlength{\unitlength}{\unitlength * \real{\svgscale}}%
    \fi%
  \else%
    \setlength{\unitlength}{\svgwidth}%
  \fi%
  \global\let\svgwidth\undefined%
  \global\let\svgscale\undefined%
  \makeatother%
  \begin{picture}(1,0.59039221)%
    \lineheight{1}%
    \setlength\tabcolsep{0pt}%
    \put(0,0){\includegraphics[width=\unitlength,page=1]{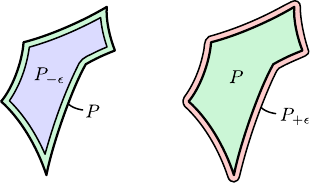}}%
  \end{picture}%
\endgroup%

	}
	\vspace{1em}
	
	\begin{lemma}\label{volume_converges}
		As $\epsilon \to 0$, the volumes of $P_{-\epsilon}$ and $P_{+\epsilon}$ both converge to $\vol(P)$.
	\end{lemma}
	
	\begin{proof}
		Decompose $P$ into finitely many convex pieces and take the finite set of hyperplanes defining these pieces. Then the region $P_{+\epsilon} - P$ is contained in the union of finitely many cylinders of the form $D^{n-1} \times [0,\epsilon]$ from \cref{cylinder_volume}. Since the volume of each of these cylinders goes to zero as $\epsilon \to 0$, so too does the volume of $P_{+\epsilon} - P$. The argument for $P - P_{-\epsilon}$ is similar.
	\end{proof}
	
	We next use a bit of hyperbolic trigonometry to analyse the behaviour of thin almost-rectangles in $H^2$.
	
	\vspace{1em}
	\centerline{
	\def\svgwidth{2.3in}
\begingroup%
  \makeatletter%
  \providecommand\color[2][]{%
    \errmessage{(Inkscape) Color is used for the text in Inkscape, but the package 'color.sty' is not loaded}%
    \renewcommand\color[2][]{}%
  }%
  \providecommand\transparent[1]{%
    \errmessage{(Inkscape) Transparency is used (non-zero) for the text in Inkscape, but the package 'transparent.sty' is not loaded}%
    \renewcommand\transparent[1]{}%
  }%
  \providecommand\rotatebox[2]{#2}%
  \newcommand*\fsize{\dimexpr\f@size pt\relax}%
  \newcommand*\lineheight[1]{\fontsize{\fsize}{#1\fsize}\selectfont}%
  \ifx\svgwidth\undefined%
    \setlength{\unitlength}{154.70996324bp}%
    \ifx\svgscale\undefined%
      \relax%
    \else%
      \setlength{\unitlength}{\unitlength * \real{\svgscale}}%
    \fi%
  \else%
    \setlength{\unitlength}{\svgwidth}%
  \fi%
  \global\let\svgwidth\undefined%
  \global\let\svgscale\undefined%
  \makeatother%
  \begin{picture}(1,0.39508691)%
    \lineheight{1}%
    \setlength\tabcolsep{0pt}%
    \put(0,0){\includegraphics[width=\unitlength,page=1]{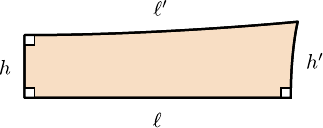}}%
  \end{picture}%
\endgroup%

	}
	\vspace{1em}
	
	\begin{lemma}\label{thin_rectangle}
	Given a figure as above in $H^2$ with all four sides geodesics, if we fix $\ell > 0$ and regard the other quantities as functions of $h$, then to first order in $h$,
		\[ \ell' \sim \ell \qquad \textup{and} \qquad h' \sim h \cosh \ell. \]
	\end{lemma}
	
	\begin{proof}
		When $h = 0$, $\ell' = \ell$ and $h' = 0$. Therefore the first-order approximation has the form
		\[ \ell' \sim \ell + ah, \qquad h' \sim bh. \]
		We draw a segment $c$ from the top-left to bottom-right of the rectangle, and establish $\cosh c = \cosh \ell\cosh h$ using the hyperbolic law of cosines. Then we use the hyperbolic law of cosines again on the top-right triangle to express $\ell'$ in terms of $c$, $h'$, and the cosine of the lower-right angle, which in turn can be rewritten as $\frac{\sinh h}{\sinh c}$. The $\sinh c$ cancels out and we can substitute $\cosh c = \cosh \ell\cosh h$, giving
		\[ \cosh \ell' = \cosh \ell \cosh h \cosh h' - \sinh h \sinh h'. \]
		Plugging in the first-order approximations for $\ell'$ and $h'$ gives
		\[ \cosh \ell + ah\sinh \ell \sim (\cosh \ell)(1)(1) - (h)(bh), \]
		which implies that $ah \sinh \ell \sim 0$ and therefore $a = 0$ as desired. A similar procedure with the hyperbolic law of sines on the upper-right triangle gives
		\[ \sinh h' \tanh \ell  = \sinh \ell' \tanh h, \]
		which on first-order approximations becomes
		\[ (bh)(\tanh \ell) \sim (\sinh \ell)(h), \]
		and therefore $b = \cosh \ell$ as desired.
	\end{proof}
	
	Next we use this to cut up a ball in hyperbolic space into polytopes of arbitrarily small diameter. This is a replacement for the step of the proof of \cref{euclidean_ea} where we cut Euclidean space into cubes $C_i$. Let $B_r(x)$ denote the ball of radius $r$ about $x \in H^n$.
	
	\begin{lemma}\label{flow_polytopes_1}
		For each $r > 0$ we can find a finite set of polytopes $\{C_i\}_{i \in I}$ in $H^n$ such that
		\begin{itemize}
			\item The $C_i$ have pairwise disjoint interiors,
			\item $B_r(x) \subseteq \bigcup_{i \in I} C_i \subseteq B_{2r}(x)$, and
			\item the diameters of the pieces $C_i$ are as small as desired.
		\end{itemize}
	\end{lemma}
	
	\begin{lemma}\label{flow_polytopes_2}
		For each $r > 0$ and $\epsilon > 0$, we can find two finite sets of polytopes $\{C_i\}_{i \in I}$ and $\{D_j\}_{j \in J}$ in $H^n$ satisfying all the conditions of the previous lemma, and such that
		\begin{itemize}
			\item for each $i \in I$ and $j \in J$ we have an orientation-preserving isometry $g_{ij}$ such that $g_{ij}(C_i)$ is contained in the interior of $D_j$, and
			\item the maximum ratio of the volumes satisfies the bound
			\[ \max_{i,j} \left( \frac{\vol(D_j)}{\vol(C_i)} \right) < (\cosh (2r))^{2n-2} + \epsilon. \]
		\end{itemize}
	\end{lemma}
	
	\begin{proof}
		We begin with the proof of the first lemma; the second lemma will use the same construction.
		
		We induct on the dimension of our hyperbolic space $H^n$. The case where $n = 1$ is easy. For general $n$, we fix a geodesic $\gamma$ passing through $x$, and consider equally-spaced hyperplanes $H_\alpha$ perpendicular to $\gamma$, with arbitrarily small distance $h$ between them.

	\vspace{1em}
	\centerline{
	\def\svgwidth{2.5in}
\begingroup%
  \makeatletter%
  \providecommand\color[2][]{%
    \errmessage{(Inkscape) Color is used for the text in Inkscape, but the package 'color.sty' is not loaded}%
    \renewcommand\color[2][]{}%
  }%
  \providecommand\transparent[1]{%
    \errmessage{(Inkscape) Transparency is used (non-zero) for the text in Inkscape, but the package 'transparent.sty' is not loaded}%
    \renewcommand\transparent[1]{}%
  }%
  \providecommand\rotatebox[2]{#2}%
  \newcommand*\fsize{\dimexpr\f@size pt\relax}%
  \newcommand*\lineheight[1]{\fontsize{\fsize}{#1\fsize}\selectfont}%
  \ifx\svgwidth\undefined%
    \setlength{\unitlength}{180.13528095bp}%
    \ifx\svgscale\undefined%
      \relax%
    \else%
      \setlength{\unitlength}{\unitlength * \real{\svgscale}}%
    \fi%
  \else%
    \setlength{\unitlength}{\svgwidth}%
  \fi%
  \global\let\svgwidth\undefined%
  \global\let\svgscale\undefined%
  \makeatother%
  \begin{picture}(1,0.85119926)%
    \lineheight{1}%
    \setlength\tabcolsep{0pt}%
    \put(0,0){\includegraphics[width=\unitlength,page=1]{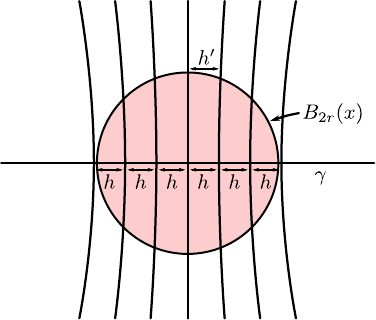}}%
  \end{picture}%
\endgroup%

	}
	\vspace{1em}
	
		We constrain ourselves to the ball $B_{2r}(x)$, and let $h'$ be the maximum distance over this ball that it takes to travel from any plane $H_\alpha$ to the next plane $H_{\alpha + 1}$, along a geodesic perpendicular to $H_\alpha$. By \cref{thin_rectangle}, we get that $h' \lesssim h \cosh (2r)$, for small $h$ and fixed $r$.

\renewcommand*{\figwidth}{1.4in}
\hspace{-1.5em}
\begin{minipage}[t]{\dimexpr\textwidth-\figwidth-\figmargin}
	
\setlength{\parindent}{1.5em}
	We take any polytope $P \subseteq H_\alpha$ and form the geodesic flow in the perpendicular direction to $H_\alpha$, giving a map
		\[ \Phi\colon P \times \R \to H^n. \]
		The map $\Phi$ is not an isometry, but every copy of $\R$ does go to a geodesic. It follows that if we intersect the image of $P \times \R$ with the space between the hyperplanes $H_\alpha$ and $H_{\alpha+1}$, we get a polytope that we call the \emph{flow polytope} $f(P,h)$.	By our definition of $h$ and $h'$ we get the inclusions
		\[ \Phi(P \times [0,h]) \subseteq f(P,h) \subseteq \Phi(P \times [0,h']). \]
		The first conclusion we draw from this is that
		\[ \diam f(P,h) \leq 2h' + \diam P \lesssim 2h\cosh (2r) + \diam P. \]

\vspace{.5em}

\end{minipage}
\hfill
\begin{minipage}[t]{\dimexpr\figwidth}\vspace{0pt}
	
	\def\svgwidth{\figwidth}
\begingroup%
  \makeatletter%
  \providecommand\color[2][]{%
    \errmessage{(Inkscape) Color is used for the text in Inkscape, but the package 'color.sty' is not loaded}%
    \renewcommand\color[2][]{}%
  }%
  \providecommand\transparent[1]{%
    \errmessage{(Inkscape) Transparency is used (non-zero) for the text in Inkscape, but the package 'transparent.sty' is not loaded}%
    \renewcommand\transparent[1]{}%
  }%
  \providecommand\rotatebox[2]{#2}%
  \newcommand*\fsize{\dimexpr\f@size pt\relax}%
  \newcommand*\lineheight[1]{\fontsize{\fsize}{#1\fsize}\selectfont}%
  \ifx\svgwidth\undefined%
    \setlength{\unitlength}{95.39742901bp}%
    \ifx\svgscale\undefined%
      \relax%
    \else%
      \setlength{\unitlength}{\unitlength * \real{\svgscale}}%
    \fi%
  \else%
    \setlength{\unitlength}{\svgwidth}%
  \fi%
  \global\let\svgwidth\undefined%
  \global\let\svgscale\undefined%
  \makeatother%
  \begin{picture}(1,1.36676821)%
    \lineheight{1}%
    \setlength\tabcolsep{0pt}%
    \put(0,0){\includegraphics[width=\unitlength,page=1]{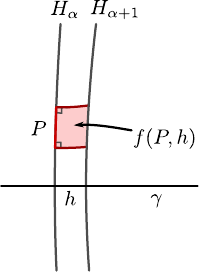}}%
  \end{picture}%
\endgroup%

\end{minipage}

		Therefore if we inductively define polytopes $C_i$ by taking the flow polytopes for a collection of $(n-1)$-dimensional polytopes covering the ball of radius $r$ in each of the hyperplanes $H_\alpha$, using the same value of $h$ each time, we get $\diam C_i \lesssim 2nh\cosh(2r)$. This goes to zero as $h \to 0$, proving the first lemma. A sketch of the polytopes $\{C_i\}$ is given below.
		
	\vspace{.8em}
	\centerline{
	\def\svgwidth{2.4in}
\begingroup%
  \makeatletter%
  \providecommand\color[2][]{%
    \errmessage{(Inkscape) Color is used for the text in Inkscape, but the package 'color.sty' is not loaded}%
    \renewcommand\color[2][]{}%
  }%
  \providecommand\transparent[1]{%
    \errmessage{(Inkscape) Transparency is used (non-zero) for the text in Inkscape, but the package 'transparent.sty' is not loaded}%
    \renewcommand\transparent[1]{}%
  }%
  \providecommand\rotatebox[2]{#2}%
  \newcommand*\fsize{\dimexpr\f@size pt\relax}%
  \newcommand*\lineheight[1]{\fontsize{\fsize}{#1\fsize}\selectfont}%
  \ifx\svgwidth\undefined%
    \setlength{\unitlength}{180.13528095bp}%
    \ifx\svgscale\undefined%
      \relax%
    \else%
      \setlength{\unitlength}{\unitlength * \real{\svgscale}}%
    \fi%
  \else%
    \setlength{\unitlength}{\svgwidth}%
  \fi%
  \global\let\svgwidth\undefined%
  \global\let\svgscale\undefined%
  \makeatother%
  \begin{picture}(1,0.85119926)%
    \lineheight{1}%
    \setlength\tabcolsep{0pt}%
    \put(0,0){\includegraphics[width=\unitlength,page=1]{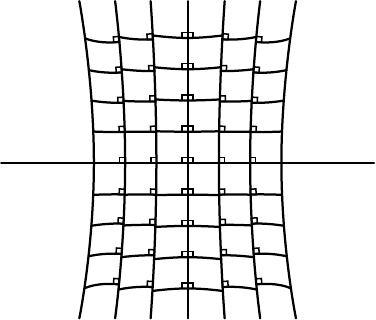}}%
  \end{picture}%
\endgroup%

	}
	\vspace{1em}
		
		For the second lemma we do the same construction to build both $C_i$ and $D_j$. We note that for any pair of hyperplanes $H_\alpha$ and $H_\beta$, any isometry of $H^n$ that moves $H_\alpha$ to $H_\beta$ will commute with the flow maps from these two hyperplanes. Therefore, if $C \subseteq H_\alpha \cap B_{2r}(x)$, $D \subseteq H_\beta \cap B_{2r}(x)$, and there an isometric embedding of $C$ into the interior of $D$ as subsets of $H^{n-1}$, then there is also an isometric embedding of $f(C,h)$ into $f(D,h')$ as subsets of $H^n$. Furthermore, if the isometry from $C$ into $D$ is oriented, the isometry from $f(C,h)$ into $f(D,h')$ is oriented as well.
		
		Therefore, if we use $h$ as the hyperplane spacing for $C_i$ and $h'$ as the hyperplane spacing for $D_j$, we get two covers of $B_r(x)$ with the property that each piece $C_i$ embeds into each piece $D_j$, as required.
		
		We then calculate the volume ratio using \cref{cylinder_volume}:
		\begin{align*}
			&h\vol(C) \lesssim \vol f(C,h) \lesssim h'\vol(C) \lesssim h\cosh(2r)\vol(C).
		\end{align*}
		If we therefore start with covers of the ball in $H^1$ by intervals, with volume differing by a factor of $(1+\delta)$ between the two covers for some $\delta > 0$, and apply this flow construction $(n-1)$ times, using $h$ for the first cover and $h'$ for the second cover each time, we get two covers of the ball in $H^n$ whose volume ratio is
		\[
			\frac{\vol(D_j)}{\vol(C_i)}
			\lesssim (1+\delta)\left(\frac{h'\cosh(2r)}{h}\right)^{n-1}
			\lesssim (1+\delta)\left(\frac{h\cosh^2(2r)}{h}\right)^{n-1}
			\lesssim (1+\delta)\left(\cosh(2r)\right)^{2n-2}.
		\]
		We can therefore make this volume ratio as close to $\left(\cosh(2r)\right)^{2n-2}$ as we want by making $h$ and $\delta$ small. Furthermore, the factor of $(1+\delta)$ and the inductive argument ensure that each $C_i$ embeds into the interior of each $D_j$ by an oriented isometry.
	\end{proof}
	
	\begin{lemma}\label{equal_ratio_pieces}
		Given two hyperbolic polytopes $P$ and $Q$, letting $m = \frac{\vol(Q)}{\vol(P)}$, we can cut $P$ and $Q$ into an equal number of pieces $P_1$, $\ldots$, $P_k$ and $Q_1$, $\ldots$, $Q_k$ such that 
		\begin{itemize}
			\item $\vol(Q_i) = m\vol(P_i)$ for $i = 1, \ldots, k$, and
			\item the diameters $\diam P_i$ and $\diam Q_i$ are as small as desired.
		\end{itemize}
	\end{lemma}
	
	\begin{proof}
		We first cut them into pieces of small diameter by intersecting with the polytopes $C_i$ from \cref{flow_polytopes_1}. Then we make the volume ratios equal by moving through the given lists one element at a time: if $m\vol(P_1) > \vol(Q_1)$, we cut off a piece from $P_1$ of volume $\frac{1}{m}\vol(Q_1)$ and set this piece and $Q_1$ to the side. Otherwise, we cut off a piece from $Q_1$ of volume $m\vol(P_1)$ and set this piece and $P_1$ to the side. Both of these moves are possible using a single hyperplane cut -- the intermediate value theorem guarantees that a hyperplane exists cutting the polytope into pieces of the desired volume. Repeating this step finitely many times, we achieve two lists of pieces with the desired property.
	\end{proof}
	
	\begin{theorem}\label{hyperbolic_ea}
	For $n>0$, the assembler $\cH^n_G$ is an EAZ-assembler, with volume function given by hyperbolic volume, provided $G \leq I(H^n)$ and $G$ contains a dense subgroup of the group of orientation-preserving isometries.
	\end{theorem}

	\begin{proof}
	As in the proof of \cref{equal_ratio_pieces}, a polytope $P$ can be cut along hyperplanes to form pieces of any desired volume, provided the volumes sum to $\vol(P)$, by a simple application of the intermediate value theorem. This is enough to prove \ref{enum:ass-vol-archimedean}.
	
	Property \ref{enum:ass-vol-zylev} is a consequence of \cref{eaz_reduction} once we have checked \ref{enum:ass-vol-existence}, because \ref{enum:ass-vol-iso-0} again holds by definition: the only $n$-polytope in $\cH^n_G$ with $\vol(P) = 0$ is the empty $n$-polytope.
	
	To prove \ref{enum:ass-vol-existence}, we consider any $P$ and $Q$ such that $m = \frac{\vol(Q)}{\vol(P)} > 1$. Using \cref{equal_ratio_pieces} we cut $P$ and $Q$ into pieces where the volume ratio is still $m$, but now the diameter $r$ of the pieces is so small that
	\[ (\cosh(2r))^{2n-2} < m^{1/3}. \]
	It suffices to prove the embedding for each piece separately, so we assume now that $P$ and $Q$ themselves have diameter less than $r$. Without loss of generality they are both subsets of $B_r(x)$.
	
	We next use \cref{volume_converges} to choose $\epsilon > 0$ such that
	\[ \frac{\vol(P_{+\epsilon})}{\vol(P)} < m^{1/3} \quad \textup{ and } \quad \frac{\vol(Q)}{\vol(Q_{-\epsilon)}} < m^{1/3},
	\quad \textup{ which implies } \quad  \frac{\vol(Q_{-\epsilon})}{\vol(P_{+\epsilon})} > m^{1/3}. \]
	We then use \cref{flow_polytopes_2} again to create two covers of $B_r(x)$ where the diameter of each piece is less than $\epsilon$, and the ratio of volumes is so close to $(\cosh(2r))^{2n-2}$ that it is also less than $m^{1/3}$:
	\[ \max_{i,j} \left( \frac{\vol(D_j)}{\vol(C_i)} \right) < m^{1/3}. \]
	We restrict attention to those polytopes $C_i$ in this cover that have nontrivial intersection with $P$, and those $D_j$ that are entirely contained in $Q$, as we did with the cubes in the proof of \cref{euclidean_ea}. Since the diameter of each $C_i$ and $D_j$ is smaller than $\epsilon$, we get
	\[ P \subseteq \bigcup_i C_i \subseteq P_{+\epsilon}, \qquad Q_{-\epsilon} \subseteq \bigcup_j D_j \subseteq Q. \]
	Therefore the ratio of volumes for these pieces satisfies
	\[
		\frac{\sum_j \vol(D_j)}{\sum_i \vol(C_i)}
		> \frac{\vol(Q_{-\epsilon})}{\vol(P_{+\epsilon})}
		> m^{1/3}.
	\]
	Since each $D_j$ has volume no more than $m^{1/3}\vol(C_i)$, this implies there are more pieces $D_j$ in this cover than there are pieces $C_i$. Since each $C_i$ embeds into the interior of each $D_j$, we can therefore cut $P$ into the polytopes $P \cap C_i$, and embed each of these into a different piece $D_j \subseteq Q$. This achieves the desired scissors embedding of $P$ into $Q$.
	
	Since each $C_i$ embeds into the interior of each $D_j$, any isometry sufficiently close to the given oriented isometry $g\colon C_i \to D_j$ also sends $C_i$ into $D_j$. We can therefore select $g$ to be in any given dense subgroup of the oriented isometry group, and the result still holds.

	\end{proof}
	
	\begin{theorem}\label{spherical_ea} For $n>0$, the assembler $\category{S}^n_G$ is an EAZ-assembler, with volume function given by spherical volume, provided $G$ contains a dense subgroup of the group of orientation-preserving isometries.
	\end{theorem}
	
	\begin{proof}
		The proof follows the same pattern as \cref{hyperbolic_ea}. The main differences are that:
		\begin{itemize}
			\item In \cref{thin_rectangle}, we assume that $\ell < \frac\pi{2}$, and show that $\ell' \sim \ell$ and $h' \sim h \cos \ell$.
			\item In \cref{flow_polytopes_1} we assume that $r < \frac\pi{4}$ so that $B_{2r}(x)$ is inside an open hemisphere.
			\item In the proof of \cref{flow_polytopes_1} we take $h'$ to be the minimum distance from each hyperplane $H_\alpha$ to the next one, rather than the maximum distance, and consequently use $h'$ for the hyperplane spacing for the smaller cover and $h$ for the larger cover.
			\item The volume estimate in \cref{flow_polytopes_2} is $(\cos(2r))^{-(2n-2)}$, so in the proof of \cref{hyperbolic_ea} we assume that \emph{this} quantity is $< m^{1/3}$.
			\item In the proof of \cref{equal_ratio_pieces} we first cut $P$ and $Q$ into convex pieces each lying in an open hemisphere of $S^n$.\qedhere
		\end{itemize}
	\end{proof}
	\sloppy
	\printbibliography
\end{document}